\documentclass[twoside,11pt]{article}

\usepackage{graphicx, subfigure, caption}
\usepackage[top=1in,bottom=1in,left=1in,right=1in]{geometry}
\usepackage[sort&compress]{natbib} 
\usepackage{amsfonts, amsmath, amssymb, amsthm}
\usepackage{bbm}

\usepackage{color}
\definecolor{darkred}{RGB}{100,0,0}
\definecolor{darkgreen}{RGB}{0,100,0}
\definecolor{darkblue}{RGB}{0,0,150}

\usepackage{url}
\usepackage{hyperref}
\hypersetup{colorlinks=true, linkcolor=darkred, citecolor=darkgreen, urlcolor=darkblue}

\newtheorem{prp}{Proposition}
\newtheorem{lem}{Lemma}

\def\beq{\begin{equation}}
\def\eeq{\end{equation}}
\def\beqn{\begin{eqnarray*}}
\def\eeqn{\end{eqnarray*}}
\def\bitem{\begin{itemize}}
\def\eitem{\end{itemize}}
\def\benum{\begin{enumerate}}
\def\eenum{\end{enumerate}}
\def\bmult{\begin{multline*}}
\def\emult{\end{multline*}}
\def\bcenter{\begin{center}}
\def\ecenter{\end{center}}

\newcommand{\prpref}[1]{Proposition~\ref{prp:#1}}

\newcommand{\lemref}[1]{Lemma~\ref{lem:#1}}
\newcommand{\secref}[1]{Section~\ref{sec:#1}}
\newcommand{\figref}[1]{Figure~\ref{fig:#1}}

\DeclareMathOperator*{\argmax}{arg\, max}

\def\cH{\mathcal{H}}

\def\cN{\mathcal{N}}

\def\cY{\mathcal{Y}}

\def\bY{\mathbf{Y}}

\def\by{\mathbf{y}}

\def\bbR{\mathbb{R}}

\newcommand{\E}{\operatorname{\mathbb{E}}}
\renewcommand{\P}{\operatorname{\mathbb{P}}}
\newcommand{\Var}{\operatorname{Var}}
\newcommand{\Cov}{\operatorname{Cov}}

\newcommand{\pr}[1]{\mathbb{P}\left(#1\right)}

\def\iid{\stackrel{\rm iid}{\sim}}
\def\Bin{\text{Bin}}

\def\Geom{\text{Geom}}

\def\eps{\varepsilon}

\DeclareMathOperator{\sign}{sign}
\newcommand{\IND}[1]{\mathbbm{1}_{\{ #1 \}}}

\def\Smir{S^\star}

\begin{document}

\title{Distribution-Free Tests for Sparse Heterogeneous Mixtures}
\author{ 
Ery Arias-Castro %\footnote{Web: \eacurl} \
and 
Meng Wang%\footnote{Email: \href{mailto:mew023@ucsd.edu}{mew023@ucsd.edu}}
}
\date{Department of Mathematics, University of California, San Diego}
\maketitle

\begin{abstract}
We consider the problem of detecting sparse heterogeneous mixtures from a nonparametric perspective, and develop distribution-free tests when all effects have the same sign.  
Specifically, we assume that the null distribution is symmetric about zero, while the  true effects have positive median.  
We evaluate the precise performance of classical tests for the median (t-test, sign test)  and classical tests for symmetry (signed-rank, Smirnov, total number of runs, longest run tests) showing that none of them is asymptotically optimal for the normal mixture model in all sparsity regimes.
We then suggest two new tests.
The main one is a form of Higher Criticism, or Anderson-Darling, test for symmetry.
It is shown to be asymptotically optimal for the normal mixture model, and other generalized Gaussian mixture models, in all sparsity regimes.
Our numerical experiments confirm our theoretical findings.

\medskip



\end{abstract}

\section{Introduction}
\label{sec:intro}

Detecting heterogeneity in data has been an emblematic problem in statistics for decades.  We consider the following stylized variant.  We observe a sample $X_1, \dots, X_n \in \bbR$, and want to test 
\begin{align}
H_0^n : \quad & X_1, \dots, X_n \ \iid \ F(x); \label{h0} \\
H_1^n : \quad & X_1, \dots, X_n \ \iid \ (1-\eps_n) F(x) + \eps_n G(x - \mu_n). \label{h1} 
\end{align}
$F$ is the null distribution, $G$ is the non-null effects distribution, and $\eps_n \in (0,1]$ and $\mu_n > 0$ are the fraction and magnitude of the non-null (here positive) effects.

This testing problem could model a clinical trial where each one of $n$ subjects is given one of two treatments, $A$ or $B$, say for high-blood pressure, for a period of time, and then given the other treatment for another period of time.  In that setting, $X_i$ would be the decrease in blood pressure in subject $i$ under treatment $A$ minus that under treatment $B$.  The model above would be appropriate if treatment $A$ is expected to be at least as effective as treatment $B$, and strictly more effective  in a (possibly small) fraction of the subjects.  The model may also be relevant in a multiple testing situation where the $i$th test rejects for large values of the statistic $X_i$.  For example, in a gene expression experiment comparing a treatment and control group, a test statistic is computed for each gene; typically, the fraction of genes that are differentially expressed --- which corresponds to non-null effects --- is presumed to be small.

When the model $(F,G,\eps_n,\mu_n)$ is fully known, the likelihood ratio test (LRT) is the most powerful test.  Our goal is to devise adaptive, distribution-free tests\footnote{In our context, a test is distribution-free (aka nonparametric) if its level does not depend on the null distribution $F$, as long as $F$ is continuous and symmetric about zero.} that can compete with the LRT without knowledge of the model specifics.  For this to be possible, we assume that $F$ is symmetric about zero.  Our standing assumptions are:
\renewcommand{\theenumi}{(A\arabic{enumi})}
\renewcommand{\labelenumi}{\theenumi}
\begin{enumerate}
\item \label{a1} $F$ is continuous and symmetric about zero (i.e., $F(-x) = 1-F(x)$ for all $x\in\bbR$), while $G$ is continuous and has zero median.
\end{enumerate}

We emphasize that we do not consider the null and alternative hypotheses as composite hypotheses.  
A minimax approach in that direction may require more restrictions on $F$ or $G$, and would tend to focus the problem on particularly difficult distributions to test.  

We study the testing problem \eqref{h0}-\eqref{h1} in an asymptotic setting where $n \to \infty$.  (All the limits that appear in the paper are as $n \to \infty$, unless otherwise specified.)
We focus on the situation where the fraction of positive effects $\eps_n \to 0$, distinguishing between two main asymptotic regimes:
\beq \label{regimes}
\sqrt{n} \eps_n \to \begin{cases}
\infty & \quad \text{(dense regime)} ; \\
0 & \quad \text{(sparse regime)} .
\end{cases}
\eeq

We say that a test based on a statistic $S$ is asymptotically powerful (resp.~powerless) if the total variation distance between the distribution of $S$ under the null \eqref{h0} and under the alternative \eqref{h1} tends to 2 (resp.~0) when $n \to \infty$.  (Using this terminology, we avoid specifying complicated critical values.)

\subsection{A benchmark: the generalized Gaussian mixture model} \label{sec:gg}
The normal location model is often a benchmark for assessing the power loss for using distribution-free tests about the median.  For example, the asymptotic relative efficiencies of the sign and signed-rank tests relative to the $t$-test under normality are well-known in the setting where $\eps_n = 1$ under the alternative \citep{TSH}.  Here, we evaluate the performance of a distribution-free test in a richer family of models, where $F=G$ is generalized Gaussian with parameter $\gamma > 0$, defined by its density 
\beq \label{gg-f}
f(x) \propto \exp\left(- \frac{|x|^\gamma}\gamma \right), \quad x \in \bbR.
\eeq
Note that the normal distribution corresponds to $\gamma = 2$ and the double-side exponential distribution to $\gamma=1$.

Continuing the work of \cite{MR1456646}, who characterized the behavior of the likelihood ratio in the normal mixture model where $\gamma = 2$, \cite{dj04} derived the detection boundary in the generalized Gaussian mixture model.
They parameterized $\eps_n$ as 
\beq \label{eps}
\eps_n = n^{-\beta}, \quad \text{ with } \quad 0 < \beta < 1 \quad \text{(fixed)}.
\eeq  
The dense regime corresponds to $\beta < 1/2$, while the sparse regime corresponds to $\beta > 1/2$.  \cite{dj04} focused on the sparse regime, and parameterized $\mu_n$ as 
\beq \label{mu-gamma}
\mu_n = \left( \gamma r \log n  \right)^{1/\gamma}, \quad \text{ where } 0 < r <1 \quad \text{(fixed)}.
\eeq
%They found the following.

When $\gamma > 1$, define 
\beq \label{gamma>1}
\rho_\gamma^*(\beta) = \begin{cases} 
(2^{1/(\gamma -1)} - 1)^{\gamma -1} (\beta - 1/2), & 1/2 < \beta < 1 - 2^{-\gamma/(\gamma -1)}, \\ 
(1 - (1-\beta)^{1/\gamma})^{\gamma}, & 1 - 2^{-\gamma/(\gamma -1)} < \beta < 1. 
\end{cases}
\eeq
And for $\gamma \le 1$, define
\beq \label{gamma<1}
\rho_\gamma^*(\beta) = 2 (\beta - 1/2).
\eeq
Then the curve $r = \rho_\gamma^*(\beta)$ in the $(\beta,r)$ plane is the detection boundary for this testing problem, in the sense that the LRT is asymptotically powerful (resp.~powerless) when $r > \rho_\gamma^*(\beta)$ (resp.~$<$).  If $\gamma > 1$, We call moderately sparse the regime where $1/2 < \beta < 1 - 2^{-\gamma/(\gamma -1)}$ and very sparse the regime where $1 - 2^{-\gamma/(\gamma -1)} < \beta < 1$.
See \figref{boundary} for an illustration.

\begin{figure}[h!]
\centering
\includegraphics[scale=.3]{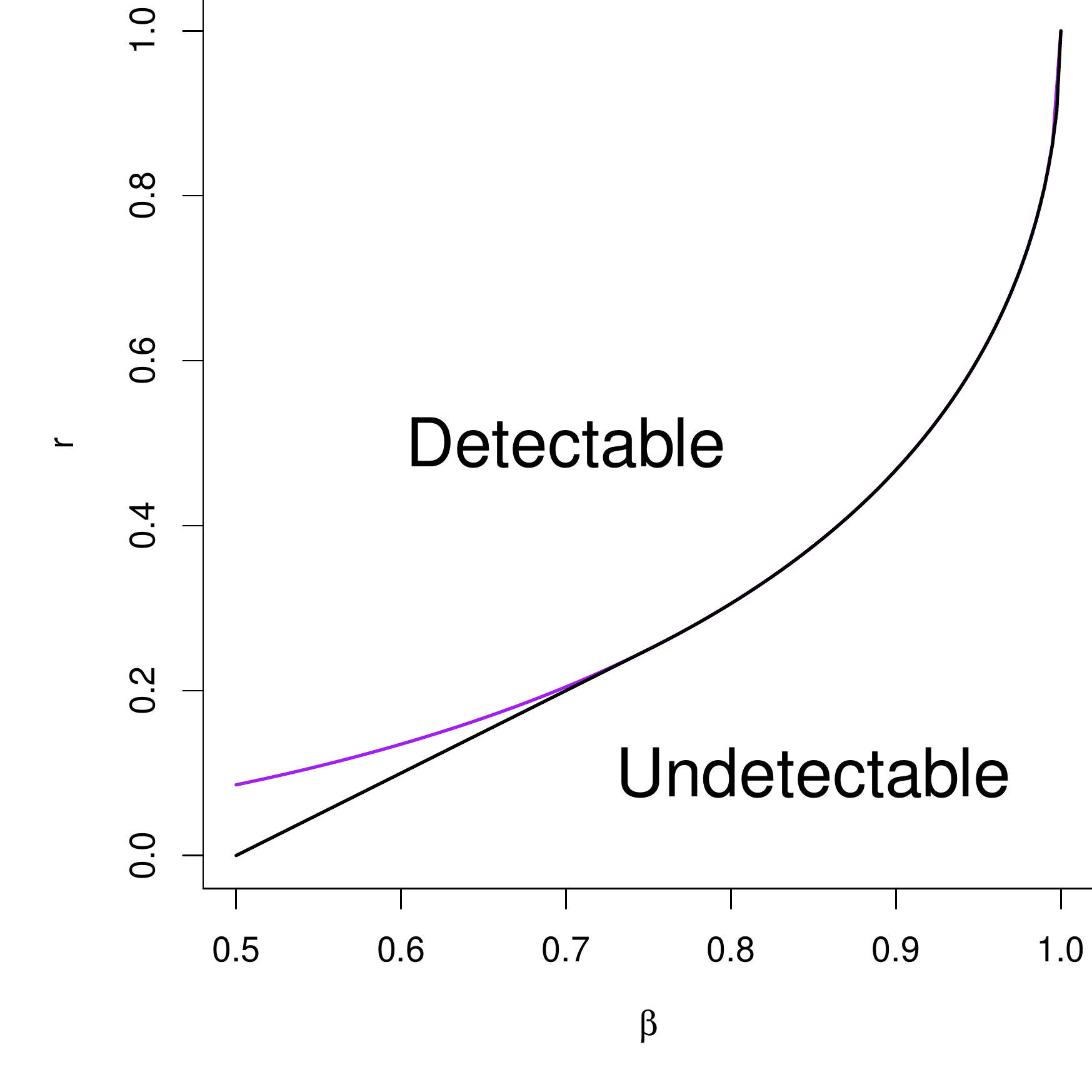}
\includegraphics[scale=.3]{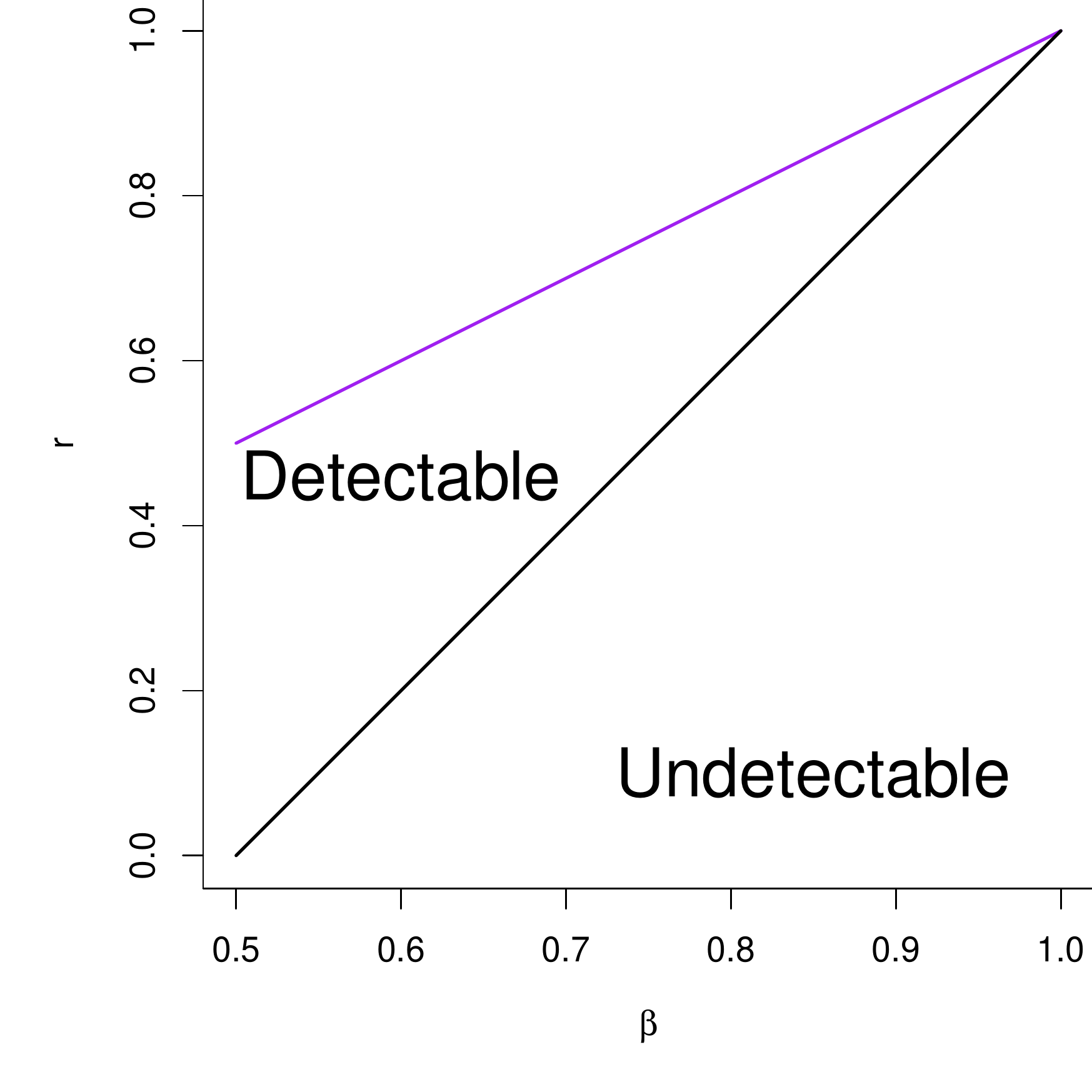}
\includegraphics[scale=.3]{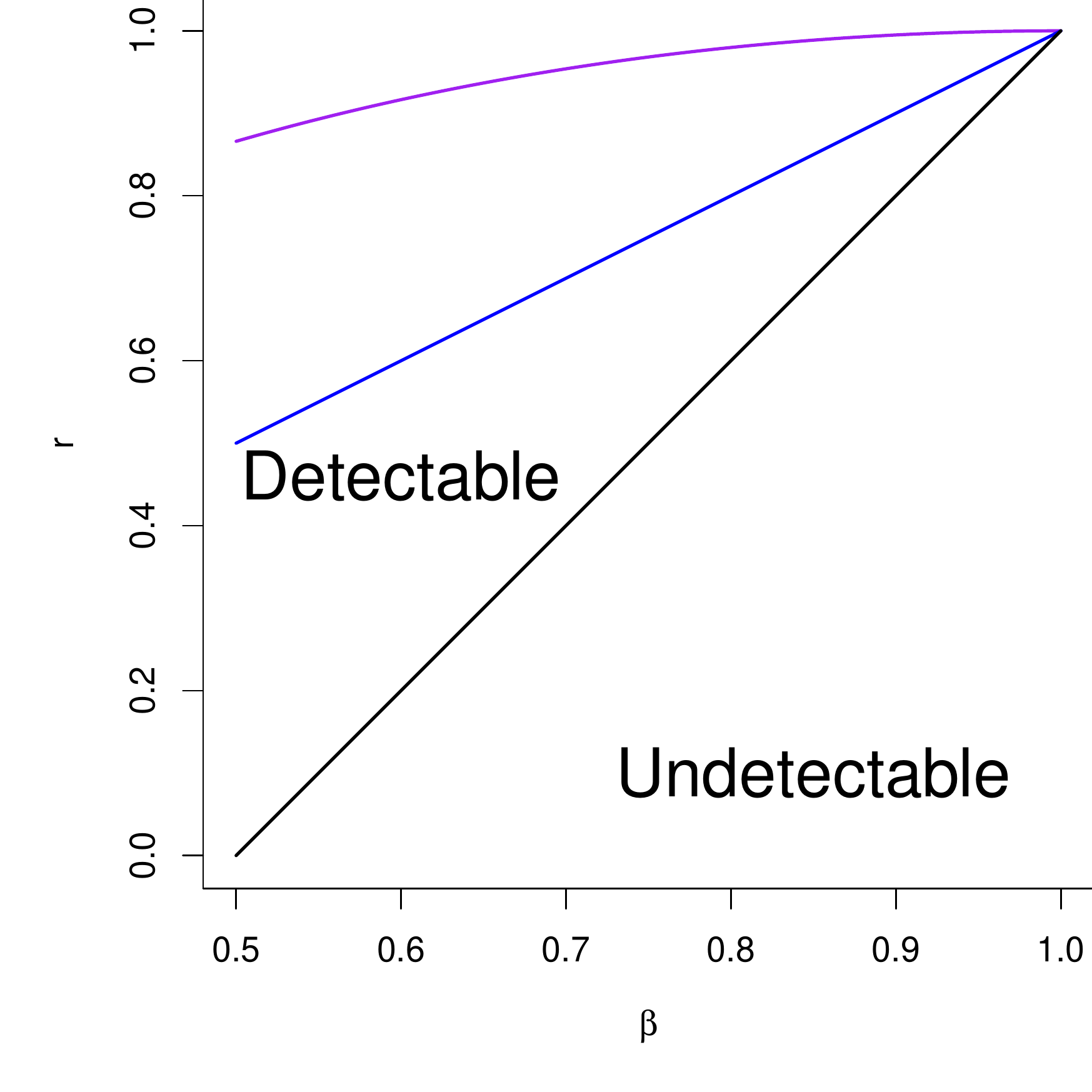}
\caption{In black is the detection boundary for the generalized Gaussian mixture model with parameter $\gamma \in \{ 2, 1, 0.5 \}$ (from left to right).  This detection boundary is attained by the CUSUM sign test.  In purple is the detection boundary for the tail-run test, and in blue is the detection boundary for the longest-run test; these coincide  when $\gamma \ge 1$.}
\label{fig:boundary}
\end{figure}

Following standard arguments, we extend these results to the dense regime, which we did not find elsewhere in the literature, except for the normal model \citep{cai2010optimal}.  

\begin{prp} \label{prp:gg}
Assume that $0 < \beta < 1/2$ and $\mu_n = n^{s-1/2}$ where $s \in (0, 1/2)$.   Then the hypotheses merge asymptotically when $\gamma \ge 1/2$ and $s < \beta$, or when $\gamma < 1/2$ and $s < \frac12 - \frac{1 - 2\beta}{1+ 2\gamma}.$
\end{prp}

\subsection{The higher criticism}
Assuming $F$ known, \cite{dj04} proposed a procedure, called higher criticism, which does not require knowledge of $G,\eps_n, \mu_n$, and is much simpler than the discretized generalized LRT proposed in \citep{ingster2002a,ingster2002b}, which still requires knowledge of $G$.  Among other things, they showed that, for any $\gamma > 0$, the higher criticism is asymptotically optimal (i.e., achieves the detection boundary) in the sense that it is asymptotically powerful when  $r > \rho_\gamma^*(\beta)$.  
Inspired by an idea of John Tukey for multiple testing, this test is based on the normalized empirical process of the $X_i$, and as such, is a special case of the goodness-of-fit test proposed by \cite{MR0050238}.  Specifically, the test rejects for large values of 
\beq \label{ad}
%{\rm HC}^+_n := 
\sup_{x \in \bbR} \, \frac{F(x) - F_n(x)}{\sqrt{F(x) (1 - F(x))}},
\eeq
where $F_n(x) = \frac1n \sum_{i=1}^n \IND{X_i \le x}$ is the empirical distribution function.

Similar results have been obtained for other mixture models (e.g., chi-squared) in \citep{dj04,JinPhD}; for a normal model where $F$ is standard normal and $G$ is normal with unknown variance \citep{cai2010optimal}; and also under dependence \citep{hj08,hj09}.
More recently, \cite{cai-12} consider the detection problem in greater generality, but still assume that the null distribution is known.  Focusing on the sparse regime where $1/2 < \beta < 1$, they derive a detection boundary by characterizing the sharp asymptotics of the Hellinger distance between the null and the alternative, and then show that the higher criticism achieves the detection boundary, without knowledge of $G,\eps_n,\mu_n$.
We also mention the work of \cite{jager}, who propose a goodness-of-fit testing approach based on $\phi$-divergences that includes the higher criticism as a special case.  We stress the fact that all these works assume that the null distribution $F$ is known.  

In the context of testing in a linear regression model with Gaussian noise, \cite{ingster2010detection} and \cite{anova-hc} discuss the case where the noise variance is unknown, corresponding here to a situation where $F$ is known to be in a parametric family.  

As far as we know, the only other publication that considers a nonparametric setting is \citep{MR2815777}, where the $X_i$'s are $t$-statistics, for example obtained from the comparison of two samples as in some gene expression analysis; there, conditions are derived under which the higher criticism is asymptotically powerful as the degree of freedom of the $t$-statistics tends to infinity.

\subsection{A new testing procedure: the CUSUM sign test}
Since $F$ is assumed to be symmetric \ref{a1}, it makes sense to test for symmetry.  We prove in this paper that none of the classical procedures are completely satisfactory in that, in the context of the normal model (for example), they do not achieve the same asymptotic performance of the LRT in all regimes, and none of them achieves the detection boundary in the moderately sparse regime.      
We propose a new test for symmetry that is satisfactory in that sense.

To better explain the rationale behind our testing procedure, we draw a parallel with the higher criticism.  In the normal model, the Kolmogorov-Smirnov test is suboptimal in the sparse regimes.  This is because the deviations of the corresponding statistic \beq \label{ks}
%{\rm KS}(F,F_n) := 
\sup_{x \in \bbR} \, [ F(x) - F_n(x) ],
\eeq
are dominated by what happens near the median of the distribution, since $\Var(F(x) - F_n(x)) = \frac1n F(x) (1- F(x))$.  Compare with the higher criticism (aka, Anderson-Darling) statistic \eqref{ad}, where each statistic in the supremum has unit variance.

The Kolmogorov-Smirnov test has an analogous test for symmetry, called the Smirnov test, based on
\beq \label{smir} 
\Smir := \sup_{x \ge 0} \, [1 - F_n(x) - F_n(-x)].
\eeq
It can be seen as comparing the positive and negative parts of the sample; or as comparing $F_n$ with its symmetrization $\frac12 \big(F_n(x) + 1 - F_n(-x)\big)$. 
The Smirnov statistic may be expressed as
\beq \label{smv}
\Smir = \max_{k = 1, \dots, n} S_k, \quad \text{where} \quad S_k:= \sum_{i=1}^k \xi_{(i)},
\eeq
in terms of the sign sequence 
\beq \label{sign-seq}
\xi_{(1)}, \dots, \xi_{(n)}, \quad \text{where} \quad \xi_{(i)} = \sign(X_{(i)}) \in \{-1,1\},
\eeq
where $|X_{(1)}| > \cdots > |X_{(n)}|$ are the observations sorted in decreasing order according their absolute value.  This sign sequence is i.i.d.~Rademacher under the null.

Our {\em cumulative sum (CUSUM) sign test} is to the Smirnov test for symmetry what the Anderson-Darling test is to the Kolmogorov-Smirnov test.  It is based on
\beq \label{cusum}
M := \max_{k = 1, \dots, n} \frac{S_k}{\sqrt{k}}.
\eeq
And indeed, under the null, $\Var(S_k/\sqrt{k}) = 1$ for all $k=1,\dots,n$.

{\bf Main result.}
A consequence of \prpref{cusum} is that the CUSUM sign test achieves the detection boundary for the generalized Gaussian mixture model (described in \secref{gg}) --- except (perhaps) in the dense regime when $\gamma < 1/2$.  (We do not know whether the higher criticism does better.)

\subsection{Content}
The remaining of the paper is organized as follows.  

%In \secref{lower}, we describe the results found in \citep{dj04,JinPhD} for generalized Gaussian mixture models in the sparse regime, provide a new result for the dense regime, and another result for the dense regime, applicable to other mixture models, which brings into focus the results we obtain later in the paper. 

In \secref{own}, we analyze the CUSUM sign test, and another new test, that we named the {\em tail-run test}, based on the first run of 1's in the sign sequence.
We will show that the tail-run test is asymptotically optimal in the very sparse regime of the generalized Gaussian model.  In fact, in our numerical experiments, the tail-run test outperforms the CUSUM sign test in that regime. 

In \secref{classical}, we analyze classical tests.  In \secref{median}, we analyze the $t$-test\footnote{When $F$ has finite variance, the $t$-test is asymptotically distribution-free as a consequence of the Central Limit Theorem.} and the sign test, considered as tests for the median.   
%These are relevant when the median is shifted towards the right under the alternative, which is the case in model \eqref{h1} under mild assumptions on $G$.
In \secref{sym}, we analyze some emblematic tests for symmetry: the Wilcoxon signed-rank test, the Smirnov test, the number-of-runs test and the longest-run test.  
The study of these classical tests in the context of our testing problem \eqref{h0}-\eqref{h1} is novel as far as we know.
We find that the $t$-test, the sign test, the signed-rank test and the Smirnov test are all asymptotically optimal in the dense regime of the generalized Gaussian model, but grossly suboptimal in the sparse regime.  We also find that the number-of-runs test is grossly suboptimal in all regimes, while the longest-run test is optimal for the generalized Gaussian model in the very sparse regime, but grossly suboptimal otherwise.   

%We display the performance of each test in Table~\ref{tab:summary} in the model where $F=G$ is generalized Gaussian with parameter $\gamma$.
%
%\begin{table}
%\caption{Detection performance for the various distribution-free tests considered in the paper, in the model where $F=G$ is generalized Gaussian with parameter $\gamma > 0$.}
%\label{tab:summary}
%\centering
%\medskip
%\def\arraystretch{1.5}
%\begin{tabular}{ c ||c | c }
% & \multicolumn{2}{$\gamma \le 1$} & \multicolumn{3}{$\gamma \le 1$} 
%\hline 
%\end{tabular}
%
%\end{table}

In \secref{numerics}, we perform numerical simulations to accompany our theoretical findings.  We focus on the generalized Gaussian mixture model.  

\secref{discussion} is a short discussion section, where we contrasts our testing problem where all effects are positive to the analogous testing problem where the effects can be negative or positive in the same experiment.

Appendix~\ref{sec:lower} contains the proof of \prpref{gg}, information bounds for the case where $F \ne G$ are both generalized Gaussian, and a more general information-theoretic lower bound for the dense regime.  Appendix~\ref{sec:perf} contains the proofs of all the performance bounds for all the tests considered in the paper.

\subsection{Notation} \label{sec:notation}
For $a,b \in \bbR$, let $a \wedge b = \min(a,b)$ and $a \vee b = \max(a,b)$.
For $x \in \bbR$, $x_+ = x \vee 0$ is the positive part.
For two sequences of reals $(a_n)$ and $(b_n)$: $a_n \sim b_n$ when $a_n/b_n \to 1$; $a_n = o(b_n)$ when $a_n/b_n \to 0$; $a_n = O(b_n)$ when $a_n/b_n$ is bounded; $a_n \asymp b_n$ when $a_n = O(b_n)$ and $b_n = O(a_n)$; $a_n \ll b_n$ when $a_n = o(b_n)$.  Finally, $a_n \approx b_n$ when $|a_n/b_n| \vee |b_n/a_n| = O(\log n)^w$ for some $w \in \bbR$.

We use similar notation with a superscript $P$ when the sequences $(a_n)$ and $(b_n)$ are random.  For instance, $a_n = O_P(b_n)$ means that $a_n/b_n$ is bounded in probability, i.e., $\sup_{n} \P(|a_n/b_n| > x) \to 0$ as $x \to \infty$.

When $X$ and $Y$ are random variables, $X \sim Y$ means they have the same distribution.  
For a random variable $X$ and distribution $F$, $X \sim F$ means that $X$ has distribution $F$.
For a sequence of random variables $(X_n)$ and a distribution $F$, $X_n \rightharpoonup F$ means that $X_n$ converges in distribution to $F$.
Everywhere, we identify a distribution and its cumulative distribution function.
For a distribution $F$, $\bar{F}(x) = 1 - F(x)$ will denote its survival function.

\section{New nonparametric tests for detecting heterogeneity}
\label{sec:own}

In this section, we study the CUSUM sign test and the tail-run test, respectively based on the statistics defined in \eqref{cusum} and \eqref{tail}.

\subsection{The cumulative sum (CUSUM) sign test}
\label{sec:cusum}

We analyze the CUSUM sign test, which rejects for large values of $M$ defined in \eqref{cusum}.  Under the null, $M \sim_P \sqrt{2\log\log(n)}$ \citep{MR0074712}.

\begin{prp} \label{prp:cusum}
Assuming \ref{a1}, the cumulative sums sign test is asymptotically powerful if either
\beq \label{prp_cusum1}
\sqrt{n} \eps_n [1/2 - G(-\mu_n)] \gg \sqrt{\log\log n};
\eeq
%{\rm or} there is $(x_n)$ such that 
%\beq \label{prp_cusum2}
%n(1-F(x_n)) \to 0, \quad n\eps_nG(-x_n-\mu_n) \to 0, \quad n \eps_n (1 - G(x_n - \mu_n)) \gg \log\log n;
%\eeq
or there is a sequence $(x_n)$ such that
\beq \label{prp_cusum3}
\frac{\sqrt{n} \eps_n [\bar G(x_n - \mu_n) - G(-x_n - \mu_n)]}{\sqrt{\bar F(x_n) + \eps_n [\bar G(x_n - \mu_n) - G(-x_n - \mu_n)]}} \gg \sqrt{\log\log n}.
\eeq
\end{prp}

Condition \eqref{prp_cusum1} is useful in the dense regime, where the CUSUM sign test behaves like the sign test (compare with \eqref{prp_sign}).  In essence, the quantity on the LHS measures (in a standardized scale) how much the positive effects in \eqref{h1} move the median away from 0.
Condition \eqref{prp_cusum3} is useful in the sparse regime, where the quantity on the LHS measures how much of a `bump' in the tail of the mixture distribution \eqref{h1}  the positive effects create.

{\em Generalized Gaussian mixture model.}
We apply \prpref{cusum} when $F=G$ is generalized Gaussian with parameter $\gamma > 0$ as in \secref{gg}.
In the dense regime $\beta < 1/2$, we use the fact that 
\[\sqrt{n} \eps_n [1/2 - F(-\mu_n)] \asymp \sqrt{n} \eps_n \mu_n = n^{- \beta + s} \gg \sqrt{\log \log n}\] 
when $s > \beta$, so that the CUSUM sign test achieves the detection boundary when $\gamma \ge 1/2$.
In the sparse regime $\beta > 1/2$, let $\mu_n = (\gamma r \log n )^{1/\gamma}$ as in \eqref{mu-gamma}. 
Note that, when $x > 0$, $\bar{F}(x) \asymp (1+x)^{1-\gamma} f(x)$, where the density $f$ is defined in \eqref{gg-f}.
We choose $x_n = (\gamma q \log n)^{1/\gamma}$ for some fixed $q \le 1$ chosen later on.  
%In what follows, we use the notation $a_n \approx b_n$ when $a_n \asymp (\log n)^w b_n$ for some $w \in \bbR$.
We then have 
\[\bar{F}(x_n) \approx  n^{-q},\]
\[\bar{F}(x_n-\mu_n) \approx  n^{- (q^{1/\gamma} - r^{1/\gamma})^\gamma},\]
and
\[F(-x_n-\mu_n) = \bar{F}(x_n + \mu_n) \approx  n^{- (q^{1/\gamma} + r^{1/\gamma})^\gamma},\]
where $\approx$ is defined in \secref{notation}. 
%When the term in \eqref{prp_cusum3} is
%\[
%\sqrt{n} \frac{\eps_n \big(\bar{F}(x_n-\mu_n) -F(-x_n - \mu_n)\big)}{\sqrt{n\bar{F}(x_n) + \eps_n \big(\bar{F}(x_n-\mu_n) -F(-x_n - \mu_n)\big)}} \asymp (\log n)^{-1/(2\gamma) -1/2} \ n^{\frac{1+q}2 -\beta - [q^{1/\gamma} - r^{1/\gamma}]^\gamma}.
%\]
Therefore, if $\Lambda_n$ denotes the LHS in \eqref{prp_cusum3}, then
\[
\Lambda_n \approx \frac{ n^{1/2-\beta - (q^{1/\gamma} - r^{1/\gamma})^\gamma} }{\sqrt{n^{-q} + n^{-\beta - (q^{1/\gamma} - r^{1/\gamma})^\gamma}}}.
\]

\bitem
\item If $\gamma \le 1$, we choose $q = r$, so that $x_n = \mu_n$.
In that case, when $r > 2 \beta -1$, we have
\[
\Lambda_n \approx \frac{n^{1/2-\beta}}{\sqrt{n^{-r} + n^{-\beta}}} \asymp  n^{\frac12 - \beta + \frac{r}2} \to \infty,
\]
using the fact that $\beta < 1$. 

\item If $\gamma > 1$, we define $r_\gamma = (1 - 2^{-1/(\gamma - 1)})^\gamma$.
If $r \ge r_\gamma$, we choose $q = 1$, in which case
\begin{align*}
\Lambda_n 
&\approx  \frac{n^{1/2-\beta - (1 - r^{1/\gamma})^\gamma}}{\sqrt{n^{-1} + n^{-\beta + (1 - r^{1/\gamma})^\gamma}}} \\
&\approx n^{1 - \beta - (1 - r^{1/\gamma})^\gamma} \wedge \ n^{\frac12 (1 - \beta - (1 - r^{1/\gamma})^\gamma)},
\end{align*}
and $1 - \beta - (1 - r^{1/\gamma})^\gamma > 0$ when $r > (1 - (1-\beta)^{1/\gamma})^\gamma$.
If $r < r_\gamma$, we choose $q = r/r_\gamma$, yielding
\begin{align*}
\Lambda_n 
&\approx \frac{n^{1/2-\beta - r (r_\gamma^{-1/\gamma} - 1)^\gamma}}{\sqrt{n^{-r/r_\gamma} + n^{-\beta - r (r_\gamma^{-1/\gamma} - 1)^\gamma}}} \\
& \approx n^{\frac{1+r/r_\gamma}2 - \beta - r (r_\gamma^{-1/\gamma} - 1)^\gamma} \wedge \ n^{\frac12 (1 - \beta - r (r_\gamma^{-1/\gamma} - 1)^\gamma)}.
\end{align*}
Both exponents are positive when the first one is, which is the case when $r > (2^{1/(\gamma-1)} -1)^{\gamma-1} (\beta - 1/2)$.
\eitem
Comparing with the information bounds obtained by \cite{dj04} and described in \secref{gg}, we see that the CUSUM sign test achieves the detection boundary for the generalized Gaussian mixture model.

\subsection{The tail-run test}
\label{sec:run}

We now consider the tail-run test, which rejects for large values of 
\def\tail{L^\ddag}
\beq \label{tail}
\tail = \max\{\ell \ge 0: \xi_{(1)} = \cdots = \xi_{(\ell)} = 1\}.
\eeq
It is closely related to the trimmed longest run test of \cite{MR2413575}.
We note that, under the null, $\tail \sim_P {\rm Geom}(1/2)$ since in that case the signs introduced in \eqref{sign-seq} are i.i.d.~Rademacher random variables.

\begin{prp} \label{prp:tail}
Assuming \ref{a1}, the tail-run test is asymptotically powerful if there exists a sequence $(x_n)$ such that
\beq \label{prp_tail1}
n \bar F(x_n)  \to 0, \ n\eps_nG(-x_n-\mu_n) \to 0, \ n\eps_n \bar G(x_n - \mu_n) \to \infty;
\eeq
it is asymptotically powerless if there exists a sequence $(x_n)$ such that
\beq \label{prp_tail2}
n\bar F(x_n) \to \infty, \quad n\eps_n \bar G(x_n - \mu_n) \to 0.
\eeq
\end{prp}

Condition \eqref{prp_tail1} says, in order, that the expected number of observations from $F$ that exceed $x_n$ tends to zero, that the expected number of observations from $G(\cdot-\mu_n)$ that are below $-x_n$ tends to zero, and that the expected number of observations from $G(\cdot-\mu_n)$ that exceed $x_n$ tends to infinity.
Clearly, this implies that the sign sequence starts with a number of pluses that diverges to infinity in probability, so that the test is asymptotically powerful.
In contrast to that, Condition \eqref{prp_tail2} implies that the first sign in the sign sequence will come from the sign of a null observation with probability tending to one, 
so that the test is asymptotically powerless.

We remark that, if $n\eps_n \bar G(x_n - \mu_n) \gg \sqrt{\log \log n}$ in \eqref{prp_tail1}, then it implies \eqref{prp_cusum3}, guaranteeing that the CUSUM sign test is asymptotically powerful.  That said, in numerical experiments, the tail-run test clearly dominates the CUSUM sign test in the very sparse regime.

{\em Generalized Gaussian mixture model.}
We apply \prpref{tail} when $F=G$ is generalized Gaussian with parameter $\gamma > 0$ as in \secref{gg}.  
We parameterize $\mu_n$ as in \eqref{mu-gamma}, namely, $\mu_n = (\gamma r \log n)^{1/\gamma}$ for some $r \in (0,1)$, and $\eps_n = n^{-\beta}$ as always.  Fix $a > 0$ and choose $x_n = (\gamma (1+a) \log n)^{1/\gamma}$.
Using the fact that $\bar{F}(x) \asymp (1+x)^{1-\gamma} f(x)$ when $x >0$, we have
\begin{align}
n \bar F(x_n)
&\approx n^{-a} \to 0 \notag \\
n\eps_n \bar F(x_n - \mu_n) 
&\approx n^{1-\beta - [(1+a)^{1/\gamma} - r^{1/\gamma}]^\gamma}. \label{tail-ggmm}
\end{align}
When $r > (1 - (1-\beta)^{1/\gamma})^\gamma$ is fixed, we may choose $a > 0$ small enough that the exponent in \eqref{tail-ggmm} is positive, implying that \eqref{prp_tail1} holds.
Comparing with the information bounds described in \secref{gg}, we see that the tail-run test achieves the detection boundary in the very sparse regime when $\gamma > 1$. 
Otherwise, it is suboptimal.  In fact, based on \eqref{prp_tail2}, we find that the detection boundary for the tail-run test is given by
\[
\rho^{\rm tail}_\gamma(\beta) = (1 - (1-\beta)^{1/\gamma})^\gamma,
\]
which is the same as that of the max test (based on $\max_i X_i$).

%We also note that $(1 - (1-\beta)^{1/\gamma})^\gamma > \beta$ for all $1/2 < \beta < 1$ when $\gamma < 1$, showing that the tail-run test is dominated by the longest-run test in that case.  (See the corresponding discussion in \secref{longest}.)

\section{Classical tests} \label{sec:classical}

In this section we study some classical tests for the median, and also some classical tests for symmetry, which are both applicable in our context.

\subsection{Tests for the median} \label{sec:median}

Under mild assumptions on $G$, for example if $G$ is strictly increasing at 0, the mixture distribution in \eqref{h1} has strictly positive median, so that we may use a test for the median to test for heterogeneity.  We study two such tests: the $t$-test and the sign test.

\subsubsection{The $t$-test}  \label{sec:$t$-test}

Remember that the $t$-test rejects for large values of 
\beq \label{t-test}
T = \frac{\sum_{i=1}^n X_i}{\sqrt{\sum_{i=1}^n (X_i - \bar{X})^2}}. %, \quad \text{ where }\bar{X} := \frac1n \sum_{i=1}^n X_i.
\eeq
The distribution of $T$ under the null is scale-free and asymptotically standard normal for all $F$ with finite second moment.  With that additional assumption, the $t$-test is asymptotically distribution-free.  (Below, we require finite fourth moments for technical reasons.)

\begin{prp} \label{prp:$t$-test}
Assume \ref{a1}, and that $F$ and $G$ have finite fourth moments.  Then the $t$-test is asymptotically powerful (resp.~powerless) if 
\beq \label{prp_$t$-test}
\sqrt{n} \eps_n [\mu_n + {\rm mean}(G)] \to \infty \quad \text{(resp.~$\to 0$)}.
\eeq
\end{prp}

In particular, in the generalized Gaussian mixture model with parameter $\gamma$, the $t$-test achieves the detection boundary in the dense regime if $\gamma \ge 1/2$, and grossly suboptimal in the sparse regime(s), where it requires that $\mu_n$ increase at least polynomially in $n$ to be powerful.

\subsubsection{The sign test}

The sign test rejects for large values of 
\beq \label{sign}
S = \sum_{i=1}^n \xi_{(i)},
\eeq
where the sign sequence is defined in \eqref{sign-seq}.
Under the null, $(S+n)/2 \sim \Bin(n, 1/2)$, since in that case $\xi_{(1)}, \dots, \xi_{(n)}$ are i.i.d.~Rademacher random variables.

\begin{prp} \label{prp:sign}
Assuming \ref{a1}, the sign test is asymptotically powerful (resp.~powerless) if 
\beq \label{prp_sign}
\sqrt{n} \eps_n [1/2 - G(-\mu_n)] \to \infty \quad \text{(resp.~$\to 0$)}.
\eeq
\end{prp}

We first note that the sign test is asymptotically powerless when $\sqrt{n} \eps_n \to 0$.  We also note that, when $G$ is differentiable at 0 with strictly positive derivative,  
\beq \label{sign1}
\sqrt{n} \eps_n(1/2 - G(-\mu_n)) \asymp \sqrt{n} \eps_n (\mu_n \wedge 1).
\eeq
Compare with \eqref{prp_$t$-test}. 
Otherwise, except for the $\sqrt{\log\log n}$ term on the RHS, \eqref{prp_sign} coincides with \eqref{prp_cusum1}, implying that, in the generalized Gaussian mixture model with parameter $\gamma$, the sign test achieves the detection boundary in the dense regime if $\gamma \ge 1/2$.

\subsection{Tests for symmetry} \label{sec:sym}

Assuming that $F$ is symmetric --- a reasonable assumption in our nonparametric setting --- places the problem in the context of testing for symmetry, which has been considerably discussed in the literature.  Beyond the signed-rank test \citep{wilcoxon1945individual}, many other methods have been proposed: there are tests based on runs statistics \citep{MR926417,MR2413575}; tests of Kolmogorov-Smirnov type \citep{MR0021260} or Cram\'er - von Mises type \citep{MR0365885,MR930523,MR1821437,MR1997030,MR0301840}; tests with bootstrap calibration \citep{MR883122,MR1126334}; tests based on kernel density estimation \citep{MR1443358,MR1861381}; tests based on trimmed Wilcoxon tests and on gaps \citep{MR675891}; tests based on measures of skewness \citep{MR1731875}; and many more.  We study a few emblematic tests for symmetry: the signed-rank test, the Smirnov test, the number-of-runs test and the longest-run test.  

Recall the definition of the sign sequence in \eqref{sign-seq}.

\subsubsection{The signed-rank test} \label{sec:wilcox}

The Wilcoxon signed-rank test \citep{wilcoxon1945individual} rejects for large values of 
\beq \label{wilcox}
W = \sum_{i=1}^n (n-i+1) \, \xi_{(i)}.
\eeq
Under the null, the distribution of $W$ is known in closed form, and $\sqrt{3/n^3} \, W$ is asymptotically standard normal \citep{MR758442}.

\begin{prp} \label{prp:wilcox}
Assume \ref{a1}, that $F$ and $G$ have densities $f$ and $g$, and that 
\beq \label{skewed}
G(x) + G(-x) \le 1, \quad \forall x \in \bbR.
\eeq
Then the signed-rank  test is asymptotically powerful (resp.~powerless) if
\beq \label{prp_wilcox}
\sqrt{n} \eps_n [\zeta_n \vee \eps_n \lambda_n]
\to \infty \quad \text{(resp.~$\to 0$)}.
\eeq
where $\zeta_n := \frac12 -\int G(x-\mu_n) f(x) {\rm d}x$ and $\lambda_n := \frac12 -\int G(-x-2\mu_n) g(x) {\rm d}x$. 
\end{prp}

Condition \eqref{skewed} prevents $G$ from being skewed to the left, and makes $\zeta_n$ and $\lambda_n$ non-negative, since
\begin{align}
\zeta_n 
&= \frac12 \int \big[1 - G(x-\mu_n) - G(-x-\mu_n) \big] f(x) {\rm d}x \label{zeta1} \\
&\ge \frac12 \int \big(G(x+\mu_n) - G(x-\mu_n) \big)f(x) {\rm d}x \ge 0, \label{zeta2}
\end{align}
where the inequality comes from \eqref{skewed} and is an equality when $G$ is symmetric about 0.
Similarly,
\begin{align*}
\lambda_n &\ge \frac12 \int \big(1 - G(x+2\mu_n) - G(-x-2\mu_n) \big) g(x) {\rm d}x \\
& \qquad + \int \big(G(x+2\mu_n) - G(x) \big) g(x) {\rm d}x \ge 0.
\end{align*}
If $G$ is symmetric,
\[
\lambda_n = \frac12 \int \big(G(x+2\mu_n) - G(x-2\mu_n) \big) g(x) {\rm d}x.
\]

We first note that the signed-rank test is asymptotically powerless when $\sqrt{n} \eps_n \to 0$ since $\zeta_n \vee \lambda_n = O(1)$.  (In fact, this is the case whether \eqref{skewed} holds or not.)
For the rest of this discussion, assume that ${\rm support}(F) = \bbR$.
\bitem 
\item When $G$ is {\em not} symmetric, $\zeta_n$ is bounded away from 0 since the first inequality in \eqref{zeta2} is strict in that case.  Hence, the signed-rank test is asymptotically powerful when $\sqrt{n} \eps_n \to \infty$.
\item When $G$ is symmetric, \eqref{skewed} and \eqref{zeta2} are equalities.  Assume that  $\mu_n = O(1)$, and suppose in addition that we may take $g$ bounded and continuous.  Then by dominated convergence, $\zeta_n \asymp \mu_n \int g(x) f(x) {\rm d}x \asymp \mu_n$ and $\lambda_n \asymp \mu_n \int g^2(x) {\rm d}x \asymp \mu_n$.  This implies that 
\[\sqrt{n} \eps_n (\zeta_n \vee \eps_n \lambda_n) \asymp \sqrt{n} \eps_n (\mu_n \wedge 1).\]  
Compare with \eqref{sign1}.
\eitem

\subsubsection{The Smirnov test}
\label{sec:smirnov}

Recall the Smirnov test \citep{MR0021260} based on the statistic $\Smir$ defined in \eqref{smir}, or equivalently in \eqref{smv}.
Under the null, $(S_k)$ is a simple symmetric random walk, so the reflection principle gives 
\[\P(\Smir \ge k) = 2 \P(S_n \ge k+1) + \P(S_n = k), \quad \text{for all integer $k\ge 0$}.\] 
In particular, $\Smir/\sqrt{n}$ is asymptotically distributed as the absolute value of the standard normal distribution. 

\begin{prp} \label{prp:smirnov}
Assuming \ref{a1}, the Smirnov  test is asymptotically powerful (resp.~powerless) if
\beq \label{prp_smirnov}
\sqrt{n} \eps_n \sup_{x \geq 0} [\bar G(x - \mu_n) - G(- x - \mu_n)] \to \infty \quad \text{(resp.~$\to 0$)}.
\eeq
\end{prp}

We first note that the Smirnov test is asymptotically powerless when $\sqrt{n} \eps_n \to 0$.
\bitem
\item When $G$ is {\em not} symmetric, the Smirnov test is asymptotically powerful when $\sqrt{n} \eps_n \to \infty$, since the supremum in \eqref{prp_smirnov} is bounded away from 0 in that case.
\item When $G$ is symmetric,
\begin{align*}
\sqrt{n} \eps_n \sup_{x \geq 0} [ \bar G(x - \mu_n) - G(- x - \mu_n) ]
\ge 2 \sqrt{n} \eps_n [1/2 - G(- \mu_n)],
\end{align*}
so the Smirnov test is at least asymptotically as powerful as the sign test in that case. Compare with \eqref{prp_sign}.
\eitem

\subsubsection{The number-of-runs test}
\label{sec:number}

The number-of-runs in the sign sequence $\xi_{(1)}, \dots, \xi_{(n)}$ is equal to $1 + R$, where 
\beq \label{runs}
R := \sum_{k=2}^n \IND{\xi_{(k)} \ne \xi_{(k-1)}},
\eeq
is the number of sign changes.
For example, $1+R = 5$ in the sequence 
\[(\xi_{(1)}, \dots, \xi_{(n)}) = (\underbrace{+,+,+,+}_{\text{a run}}, \underbrace{-,-,-}_{\text{a run}},\underbrace{+,+},\underbrace{-},\underbrace{+,+}).\] 
The number of runs test \citep{MR926417,McW} rejects for {\em small} values of $R$.  
Under the null, $R \sim \Bin(n-1, 1/2)$, since the summands in \eqref{runs} are i.i.d.~Rademacher random variables in this case.  

Here we content ourselves with a negative result showing that this test is comparatively less powerful than the other tests analyzed previously for testing heterogeneity.  
%Although we believe the result is essentially sharp, a positive result seems to require burdensome conditions that could potentially clutter the exposition.
 
\begin{prp} \label{prp:runs}
Assume \ref{a1}, and that $F$ and $G$ have densities $f$ and $g$ that are positive everywhere.
Then the number of runs test is asymptotically powerless if
\beq \label{prp_runs}
\sqrt{n} \eps_n \big(\zeta_n \wedge \eps_n \lambda_n\big) \to 0,
\eeq
where
\beq \label{prp_runs1}
\zeta_n := \int \frac{[g(x -\mu_n) - g(-x -\mu_n)]^2}{g(x -\mu_n) + g(-x -\mu_n)} {\rm d}x,
\eeq
and 
\beq \label{prp_runs2}
\lambda_n := \int \frac{[g(x -\mu_n) - g(-x -\mu_n)]^2}{f(x)} {\rm d}x.
\eeq
\end{prp}

We first note that $\zeta_n \le \int (g(x -\mu_n) + g(-x -\mu_n)) {\rm d}x = 2$, so the test is asymptotically powerless in the sparse regime $\sqrt{n} \eps_n \to 0$.

We apply \prpref{runs} when $F=G$ is generalized Gaussian with parameter $\gamma$.  Assume that $\mu_n = O(1)$.  Then $\lambda_n = O(\mu_n ^{2 \wedge (2\gamma+1)})$.  Indeed, define $a_n = \mu_n \gamma^{-1/\gamma}$.  Then
\begin{align*}
& \int \frac{[f(x -\mu_n) - f(-x -\mu_n)]^2}{f(x)} {\rm d}x \\
&\propto \int_0^\infty e^{-2 |x -a_n|^\gamma + |x|^\gamma} \big[1 - e^{|x +a_n|^\gamma - |x -a_n|^\gamma} \big]^2 {\rm d}x \\
&\propto \int_0^\infty e^{-x^\gamma + O(a_n (x \vee a_n)^{\gamma-1})} \big[1 - e^{O(a_n (x \vee a_n)^{\gamma-1})} \big]^2 {\rm d}x \\
&\le a_n^2 \int_0^\infty e^{-\frac12 x^\gamma} O(x \vee a_n)^{2\gamma-2} {\rm d}x \asymp a_n^{2 \wedge (2\gamma+1)} \asymp \mu_n^{2 \wedge (2\gamma+1)}.
\end{align*}
Hence, the test is asymptotically powerless if $\sqrt{n} \eps_n^2 \mu_n^{2 \wedge (2 \gamma + 1)} \to 0$.  This shows that, within this model, the test is much weaker than the sign test, the signed-rank test or the Smirnov test, which only require $\sqrt{n} \eps_n \mu_n \to \infty$ to be powerful.

\subsubsection{The longest-run test}
\label{sec:longest}

The length of the longest-run (of pluses) is defined as
\beq \label{L}
L = \argmax_\ell\{\exists j :  \xi_{(j+1)} = \cdots = \xi_{(j+\ell)} = 1\}.
\eeq
For example, $L = 8$ in the sequence 
\[
(\xi_{(1)}, \dots, \xi_{(n)}) = (-,-,-,\underbrace{+,+,+,+,+,+,+,+}_{\text{the longest-run}}, -,-,-,+,+,+,+,-).
\]
The longest-run test \citep{MR0004453} rejects for large values of $L$.
%As a function of $\xi_{(1)}, \cdots, \xi_{(n)}$, it is distribution-free.
The asymptotic distribution of $L$ under the null is sometimes called the Erd\H{o}s-R\'enyi law, due to early work by \cite{MR0272026}, who discovered that $L/\log n \to 1/\log 2$ almost surely.
The limiting distribution was derived later on \citep{MR972770}.

\begin{prp} \label{prp:longest}
Assume \ref{a1}, and that $F$ and $G$ have densities $f$ and $g$ that are positive everywhere.
Then the longest-run test is asymptotically powerless if there is a sequence $(x_n)$ such that
\beq \label{prp_longest1}
n \bar F(x_n) \to 0, \quad n \eps_n G(-x_n-\mu_n) \to 0, \quad n \eps_n \bar G(x_n-\mu_n) \to 0,
\eeq
and
\beq \label{prp_longest2}
\eps_n (\log n) \sup_{0 \le y \le x_n} \frac{[g(y -\mu_n) - g(-y -\mu_n)]_+}{f(y) + \eps_n g(y -\mu_n)} \to 0.
\eeq
It is asymptotically powerful if {\em either}:
\bitem
\item[(i)]
There is a sequence $(x_n)$ such that
\beq \label{prp_longest3} 
n \bar F(x_n) \to 0, \quad n \eps_n G(-x_n-\mu_n) \to 0, \quad n \eps_n \bar G(x_n-\mu_n) \gg \log n.
\eeq
\item[(ii)] There is a sequence $(x_n)$ satisfying \eqref{prp_longest1}, another sequence $(x_n')$ with $0 \le x'_n \le x_n$, as well as $a \in (0,1)$, $b < 1- \log(2-a)/\log(2)$ and $c,d > 0$ fixed, such that
\beq \label{prp_longest4}
\eps_n \inf_{x'_n \le x \le x'_n+c} \frac{g(x -\mu_n) - g(-x -\mu_n)}{f(x) + \eps_n g(x -\mu_n)} \ge a,
\eeq
and
\beq \label{prp_longest5}
\inf_{x'_n \le x \le x'_n+c} \big( f(x) + \eps_n g(x -\mu_n) \big) \ge d n^{-b}.
\eeq
\eitem
\end{prp}

Of all the classical tests that we studied, this is the only one with some power in the sparse regime.  The flip side is that it has very little power in the dense regime.  

We apply \prpref{longest} when $F=G$ is generalized Gaussian with parameter $\gamma$.  
Ignoring the $\log n$ term on the RHS, \eqref{prp_longest3} is essentially equivalent to \eqref{prp_tail1}, and as a consequence, the longest-run test is asymptotically powerful when $r > \rho^{\rm tail}_\gamma(\beta)$.
However, this is not the detection boundary
for the longest-run test in all cases.  Indeed, assume that $r > \beta$.
%Indeed, assume that $r \le \rho^{\rm tail}_\gamma(\beta)$.
Recall the parameterization of $\eps_n$ and $\mu_n$ in \eqref{eps} and \eqref{mu-gamma}, where $r < 1$ is fixed.  Choose $x_n = (t \gamma \log n)^{1/\gamma}$ where $t > 0$ is chosen below.  Then using the fact that $\bar{F}(x) \asymp (1+x)^{1-\gamma} f(x)$ when $x >0$, we have
\begin{align*}
& n [ \bar F(x_n) + \eps_n F(-x_n-\mu_n) + \eps_n \bar F(x_n-\mu_n)] \\
&\qquad \approx n [n^{-t} + n^{-\beta} n^{-(t^{1/\gamma} - r^{1/\gamma})^\gamma}] \to 0,
\end{align*}
for $t > 0$ large enough.
Let $a \in (0,1)$ be such that $\beta < 1- \log(2-a)/\log(2)$.
Then observe that, for $c > 0$ fixed,   
\[
\eps_n \min_{\mu_n-c \le x \le \mu_n+c} \frac{f(x -\mu_n) - f(-x -\mu_n)}{f(x) + \eps_n f(x -\mu_n)} \ge \eps_n \frac{f(c) - f(2\mu_n-c)}{f(\mu_n-c) + \eps_n f(0)} \sim \frac{f(c)}{f(0)},
\]
when $r > \beta$.  Choose $c > 0$ sufficiently small that $f(c)/f(0) \ge a$, so that \eqref{prp_longest4} is satisfied with $x_n' = \mu_n$.  
Finally, 
\[
\min_{\mu_n-c \le x \le \mu_n+c} \big( f(x) + \eps_n f(x -\mu_n) \big) \ge \eps_n \min_{\mu_n-c \le x \le \mu_n+c} f(x -\mu_n) \ge f(c) n^{-\beta},
\]
so that \eqref{prp_longest5} is satisfied with $d = f(c)$.  
We conclude that the test is asymptotically powerful when $r > \beta$.  This can be seen to be sharp based on \eqref{prp_longest1}-\eqref{prp_longest2}, so that the detection boundary for the longest run test is given by $r = \beta \wedge \rho^{\rm tail}_\gamma(\beta)$, meaning
\[
\rho_\gamma^{\rm long}(\beta) = 
\begin{cases} \beta, & \gamma \le 1; \\ (1 - (1-\beta)^{1/\gamma})^\gamma, & \gamma > 1.
\end{cases}
\]

\section{Numerical experiments}
\label{sec:numerics}

In this section, we perform simple simulations to quantify the finite-sample performance of each of the tests whose theoretical performance we established.
We consider the normal mixture model and some other generalized Gaussian mixture models.

In all these models, we take as benchmarks the likelihood ratio test (LRT) and the higher criticism (HC) --- we used the variant ${\rm HC}_n^+$ recommended by \cite{dj04}. The LRT is the optimal test when the models \eqref{h0}-\eqref{h1} are completely specified, meaning when $F,G,\eps_n,\mu_n$ are all known.
The HC has strong asymptotic properties under various mixture models and only requires knowledge of $F$.
All the other tests we considered are distribution-free, except for the $t$-test, which is only so asymptotically.

\subsection{Fixed sample size}
In this first set of experiments, the sample size was set at $n=10^6$. 
In the alternative, instead of a true mixture as in \eqref{h1}, we drew exactly $m := [n\eps]$ observations from $G(\cdot-\mu)$ and the other $n-m$ from $F$. 
We did so to avoid important fluctuations in the number of positive effects, particularly in the very sparse regime. 
All models were parameterized as described in \secref{gg}.  
In particular, $\eps = \eps_n = n^{-\beta}$ with $\beta \in (0,1)$ fixed, and in all cases, $\mu = \mu_n=n^{s-1/2}$ in the dense regime $\beta<1/2$.
We chose a few values for the parameter $\beta$, illustrating all regimes pertaining to a given model, while the parameter $s$ (or $r$) took values in a finer grid.  
%(These parameters are defined in \secref{gg}.)
Each situation was repeated 200 times for each test. 
We calibrated the distribution-free tests, and the $t$-test, using their corresponding limiting distributions under the null --- which was accurate enough for our purposes since the sample size $n=10^6$ is fairly large --- setting the level at 0.05.
The LRT and HC were calibrated by simulation and set at the same level.
What we report is the average empirical power --- the fraction of times the alternative was rejected.

\subsubsection{Normal mixture model}

In this model, $F=G$ is standard normal.  
The simulation results are reported in \figref{normal}.

\begin{figure}[h!]
\centering
\includegraphics[scale=.3]{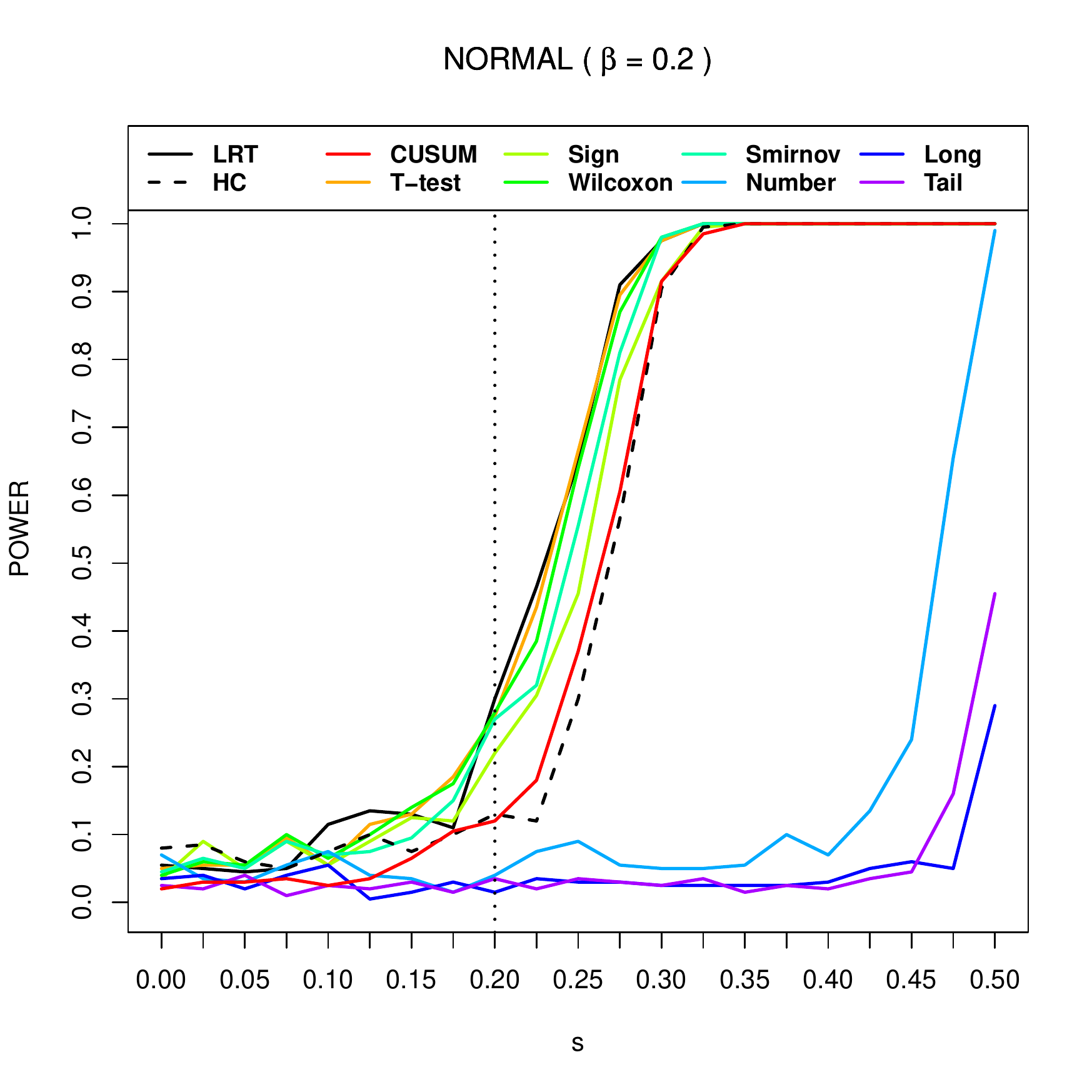}
\includegraphics[scale=.3]{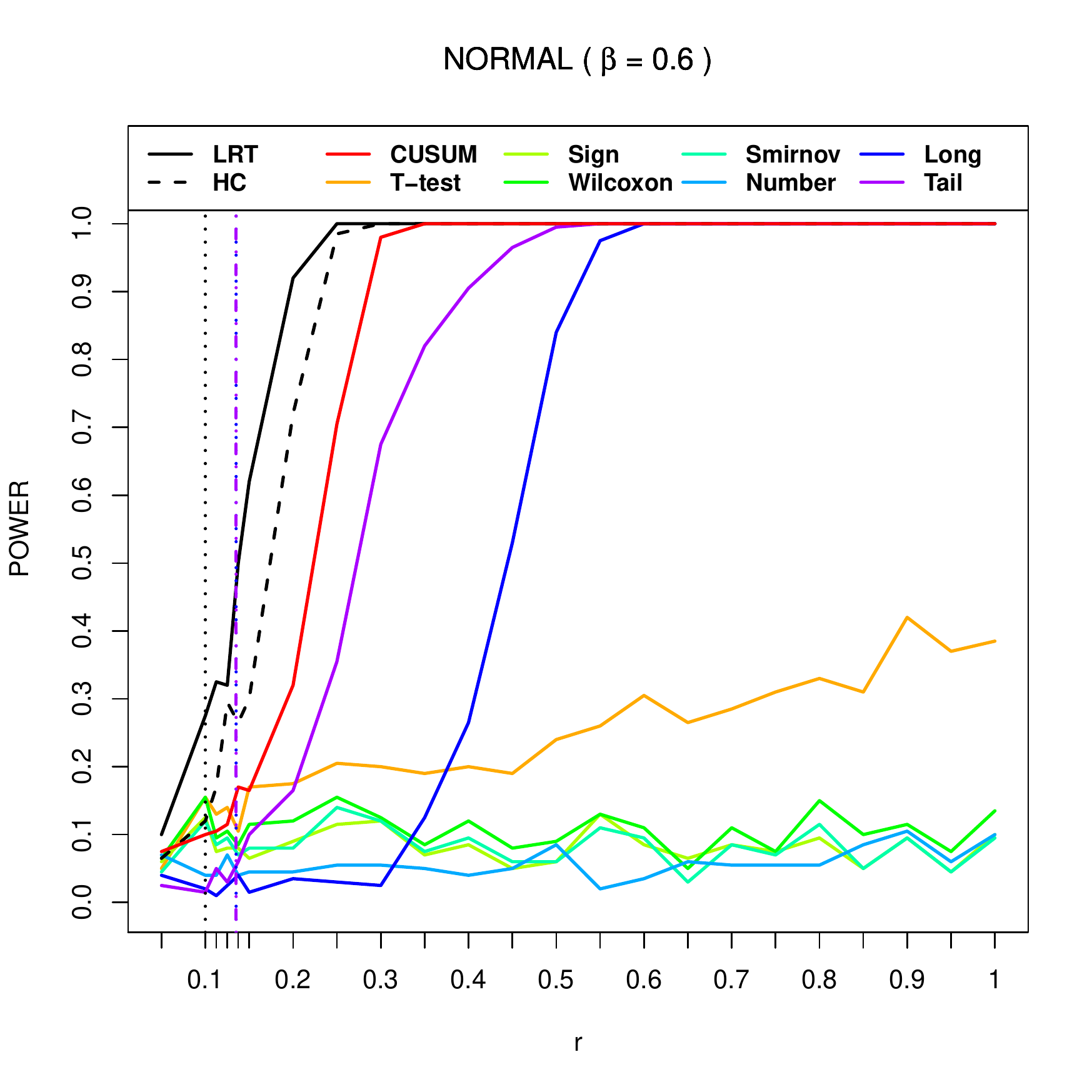}%
\includegraphics[scale=.3]{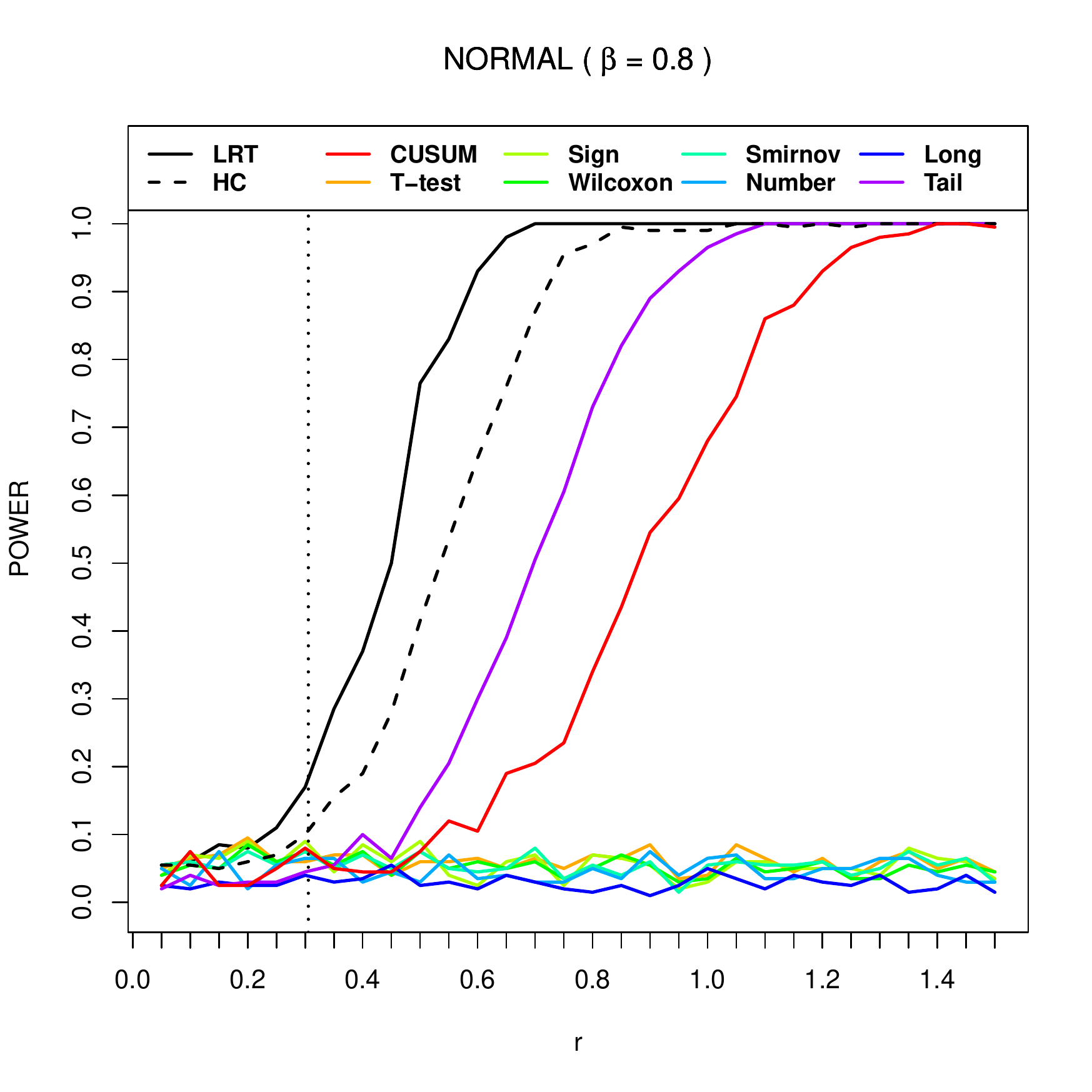}
\caption{Simulation results for the normal mixture model in three distinct sparsity regimes.
The black vertical line delineates the detection threshold.  The purple vertical line in the moderately sparse regime delineates the detection threshold for the longest-run and tail-run tests.}
\label{fig:normal}
\end{figure}

{\em Dense regime.}  We set $\beta=0.2$ and $\mu_n = n^{s-1/2}$ with $s$ ranging from 0 to 0.5 with increments of 0.025. 
From \secref{gg}, the detection threshold is at $s = \beta = 0.2$.
Moreover, the results we established in \secref{classical} imply that the $t$-test, sign, signed-rank, Smirnov tests, as well as our CUSUM sign test, all achieve this detection threshold.
The simulations are clearly congruent with the theory, will all these tests closely matching the performance of the LRT, with the HC and CUSUM sign test lagging behind a little bit.  
We also saw that the number-of-runs test is asymptotically less powerful than the aforementioned tests, and that the longest-run and tail-run tests are essentially powerless in the dense regime.  
This is obvious in the power plots.

{\em Moderately sparse regime.}
We set $\beta = 0.6$ and $\mu_n =\sqrt{2r\log(n)}$ with $r$ ranging from 0 to 1 with increments of 0.05, and added three more points equally spaced between 0.1 and 0.15 to zoom in on the phase transition.
Our theory says that all distribution-free tests are asymptotically powerless, except for the longest-run, tail-run and CUSUM sign tests, with the latter outperforming the other two.
This is indeed what happens in the simulations, although there is a fair amount of difference in power between the longest-run and tail-run tests.  The CUSUM sign test lags a little behind the HC.
The $t$-test shows some power, although not much.

{\em Very sparse regime.}
We set $\beta = 0.8$ and $\mu_n =\sqrt{2r\log(n)}$ with $r$ ranging from 0 to 1.5 with increments of 0.05. 
Our theory says that all distribution-free tests are asymptotically powerless, except for the longest-run, tail-run and CUSUM sign tests, and that all three are asymptotically near-optimal.  
In the simulations, however, the longest-run test shows no power whatsoever, and the tail-run test is noticeably more powerful than the CUSUM sign test, although quite far from the performance of the HC, which almost matches that of the LRT.
To understand what is happening, take the most favorable situation for the tail-run test, where all positives effects --- 16 of them here --- are larger than all the other observations in absolute value. 
In that case, the tail-run is of length $\tail \ge 16$, resulting in a p-value for that test smaller than $2^{-16} \approx 0.00002$.  
For the CUSUM sign test, $M \ge S_{16}/\sqrt{16} = \sqrt{16} = 4$.  But under the null, $M$ is close to $\sqrt{2\log \log n} \approx 2.3$, with deviations of about $\pm 2$ (obtained from simulations).
So even then, the number of true positives is barely enough to allow the CUSUM sign test to be fully powerful.  
As for the longest-run test, under the null, the longest-run is of length about $\log_2 n \approx 20$, with deviations of about $\pm 2$, which explains why  this test has no power.

\subsubsection{Double-exponential mixture model}

In this model, $F=G$ is double-side exponential with variance 1. 
The simulation results are reported in \figref{dexp}.

\begin{figure}[h!]
\centering
\hspace*{-.3in}
\includegraphics[scale=.3]{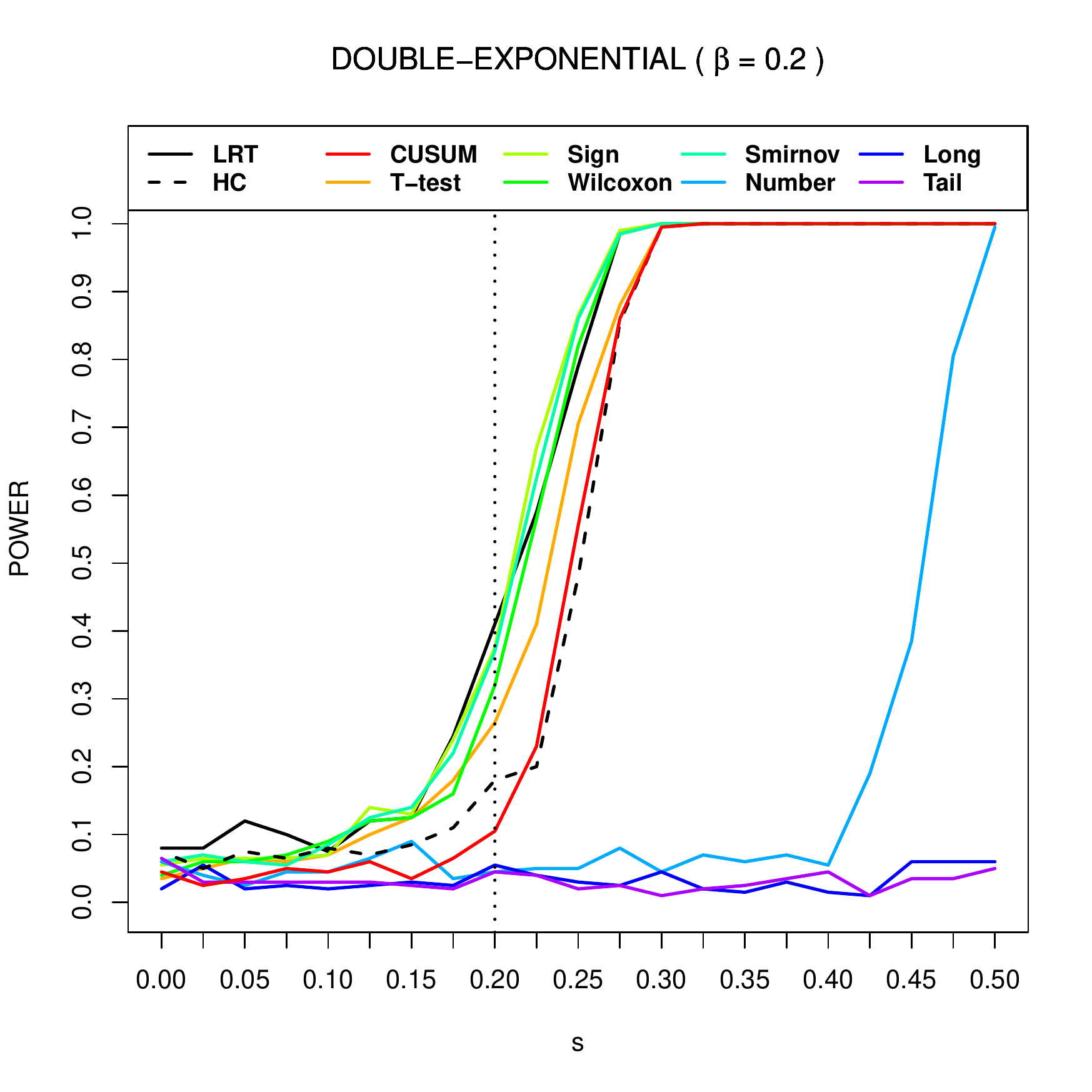}%
\includegraphics[scale=.3]{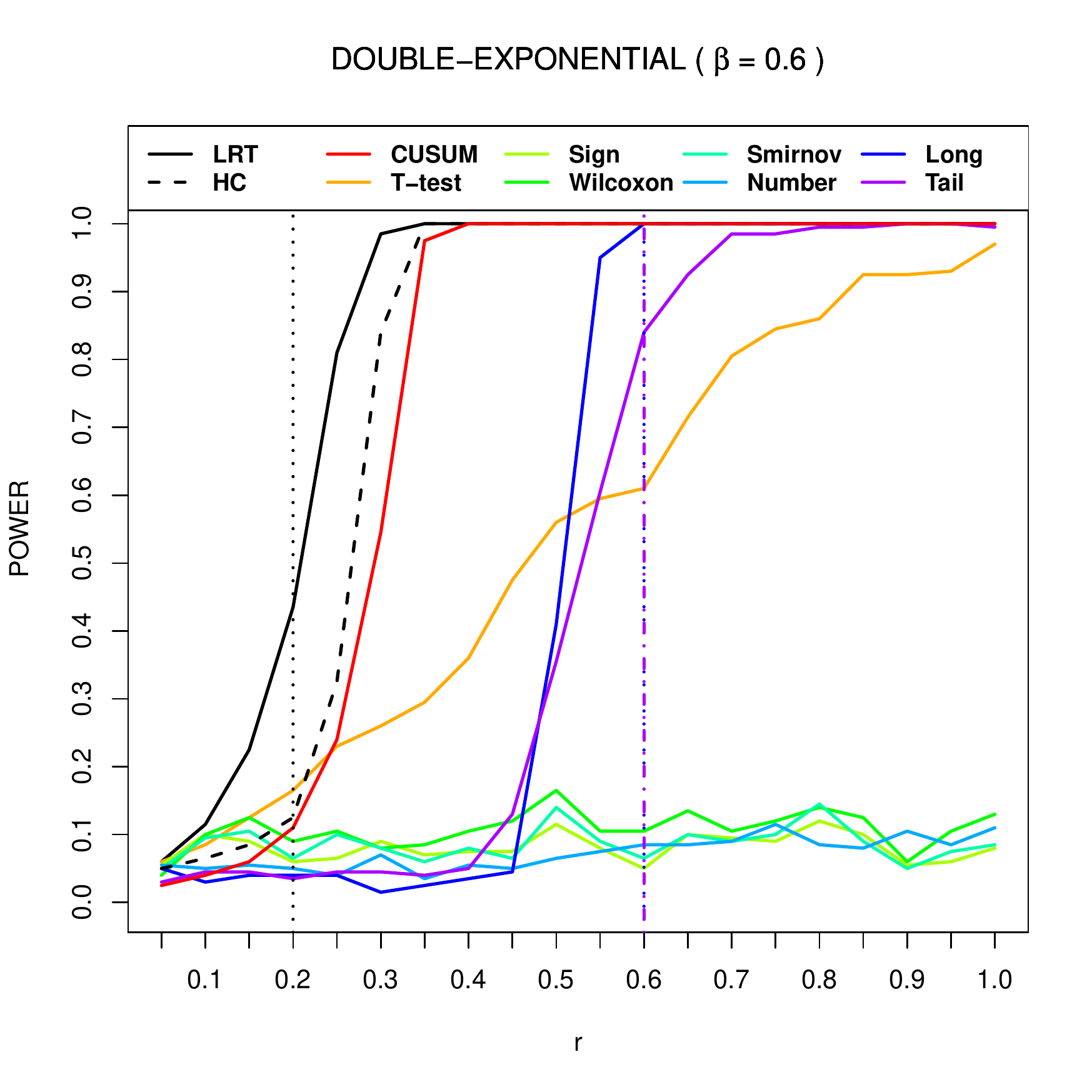}
\caption{Simulation results for the double-side exponential mixture model in the dense and sparse regimes.
The black vertical line delineates the detection threshold, while the other vertical line in the sparse regime delineates the detection threshold for the longest-run and tail-run tests.}
\label{fig:dexp}
\end{figure}

{\em Dense regime.}
The setting is exactly as in the normal mixture model.
Our theoretical findings were also similar, and are corroborated by the simulations.   

{\em Sparse regime.}
We set $\beta = 0.6$ and $\mu_n = r\log n$ with $r$ ranging from 0 to 1 with increments of 0.05.  
The simulations are congruent with the theory, with the CUSUM sign test and HC  being close in performance, while the longest-run and tail-run tests are far behind as predicted by the theory.  
The $t$-test shows a fair amount of power here, and is even fully powerful at $r =1$.
The other tests are powerless as predicted by the theory.

\subsubsection{Generalized Gaussian mixture model with $\gamma = 1/2$}

In this model, $F=G$ is generalized Gaussian with parameter $\gamma=0.5$.  
The simulation results are reported in \figref{ggmm}.

\begin{figure}[h!]
\centering
\hspace*{-.3in}
\includegraphics[scale=.3]{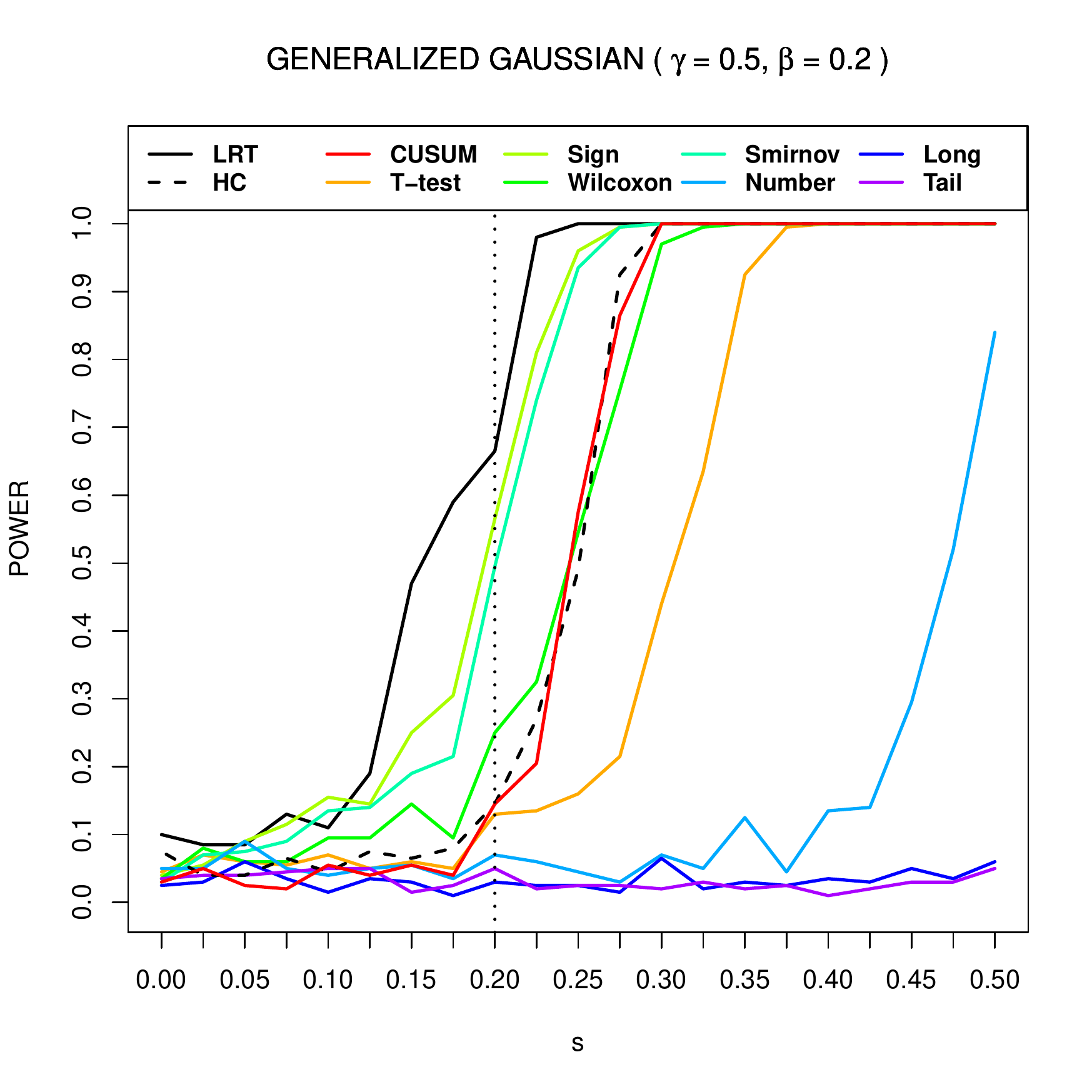}%
\includegraphics[scale=.3]{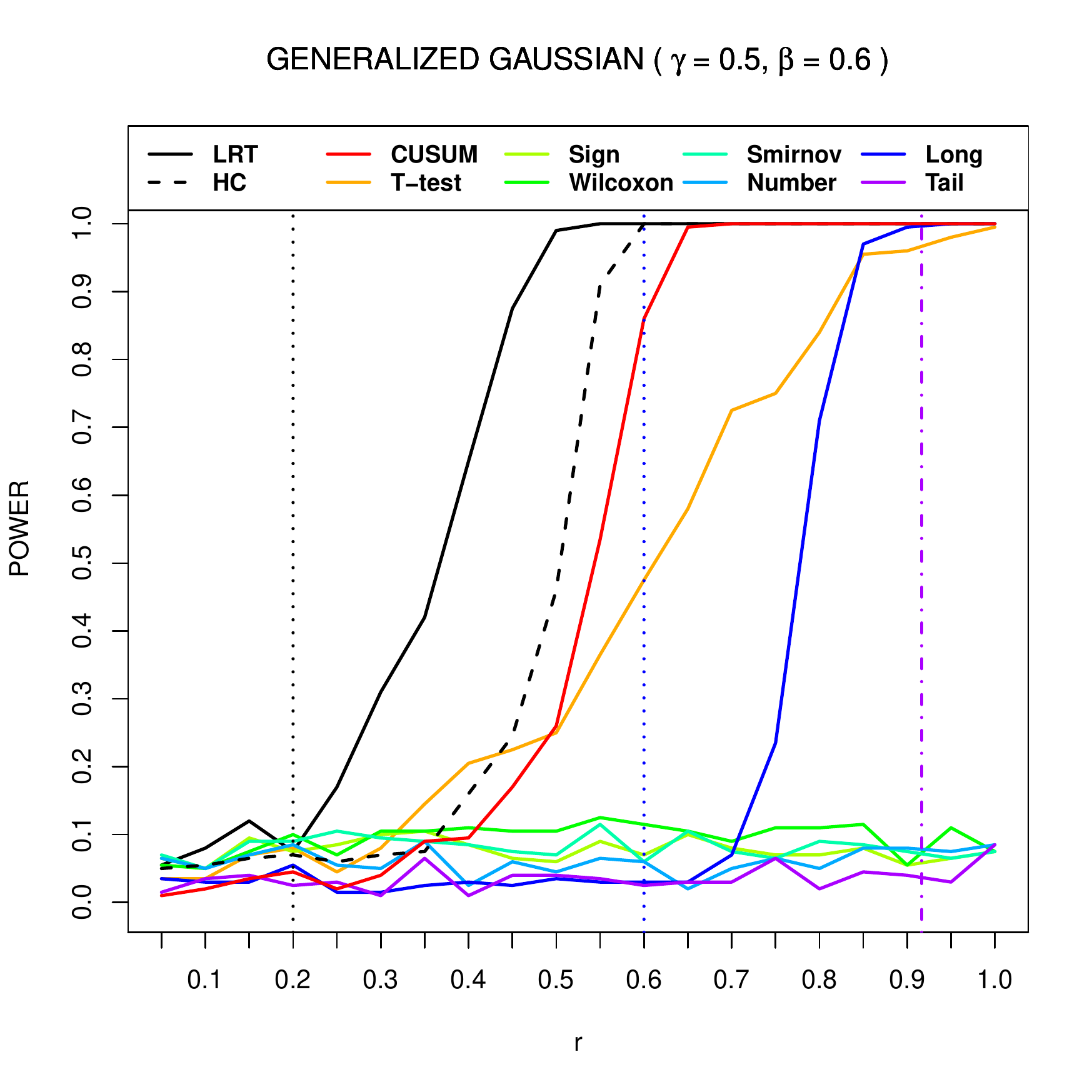}
\caption{Simulation results for the generalized Gaussian mixture model with same parameter $\gamma = 1/2$ in the dense and sparse regimes.
The black vertical line delineates the detection threshold, while the other vertical lines in the sparse regime delineate the respective detection thresholds for the longest-run and tail-run tests, according to color.}
\label{fig:ggmm}
\end{figure}

{\em Dense regime.}
The setting is exactly as in the normal mixture model.
Our theoretical findings were also similar.
The simulations illustrate the theory fairly well, although in this finite-sample situation we observe that the spread in performance is wider than before, with the best performing tests being the sign and Smirnov tests --- not far from the reigning LRT --- ahead of the signed-rank and CUSUM sign tests, very close to the HC,  
and then comes the $t$-test and number-of-runs test quite far behind.  The other two tests have no power.

{\em Sparse regime.}
We set $\beta = 0.6$ and $\mu_n = (r (\log n)/2)^2$ with $r$ ranging from 0 to 1 with increments of 0.05.  
The CUSUM sign test is slightly inferior to the HC, far above the longest-run test, as predicted by our theory.  
The tail-run test has no power here, although the theory says it should have some at $r > 1$.
The $t$-test, surprisingly, dominates the longest-run test.

\subsection{Varying sample size}

In this second set of experiments, we examined various sample sizes to assess the effect of smaller sample sizes on the power of the distribution-free tests in particular.  We focused on the CUSUM sign test and tail-run test, comparing them with the LRT and HC in the normal mixture model.  
%We did not considered the very sparse regime, in which the number of positive is too small at smaller sample sizes.
The simulation results are reported in \figref{normal2}.

\begin{figure}[h!]
\centering
\includegraphics[scale=.3]{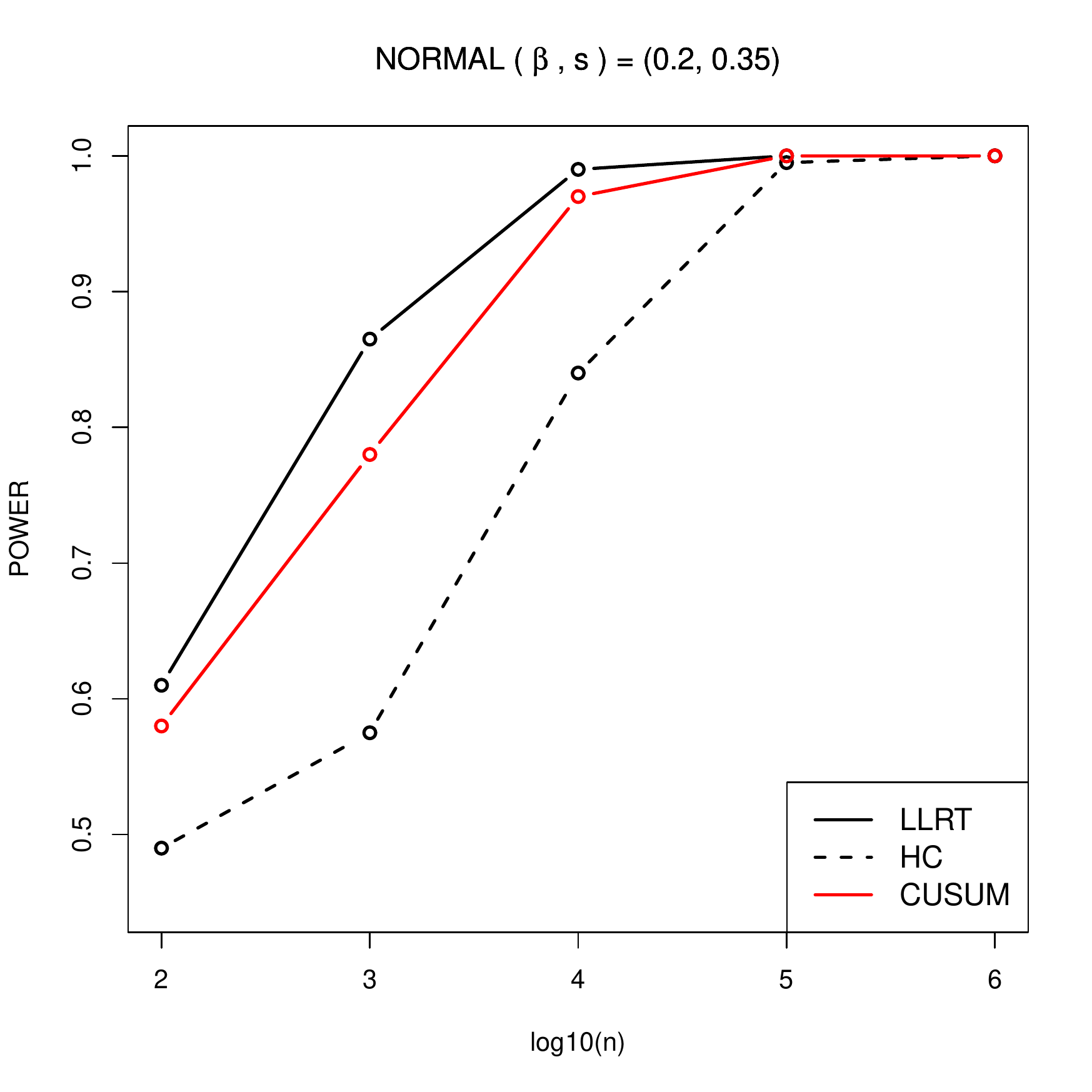}
\includegraphics[scale=.3]{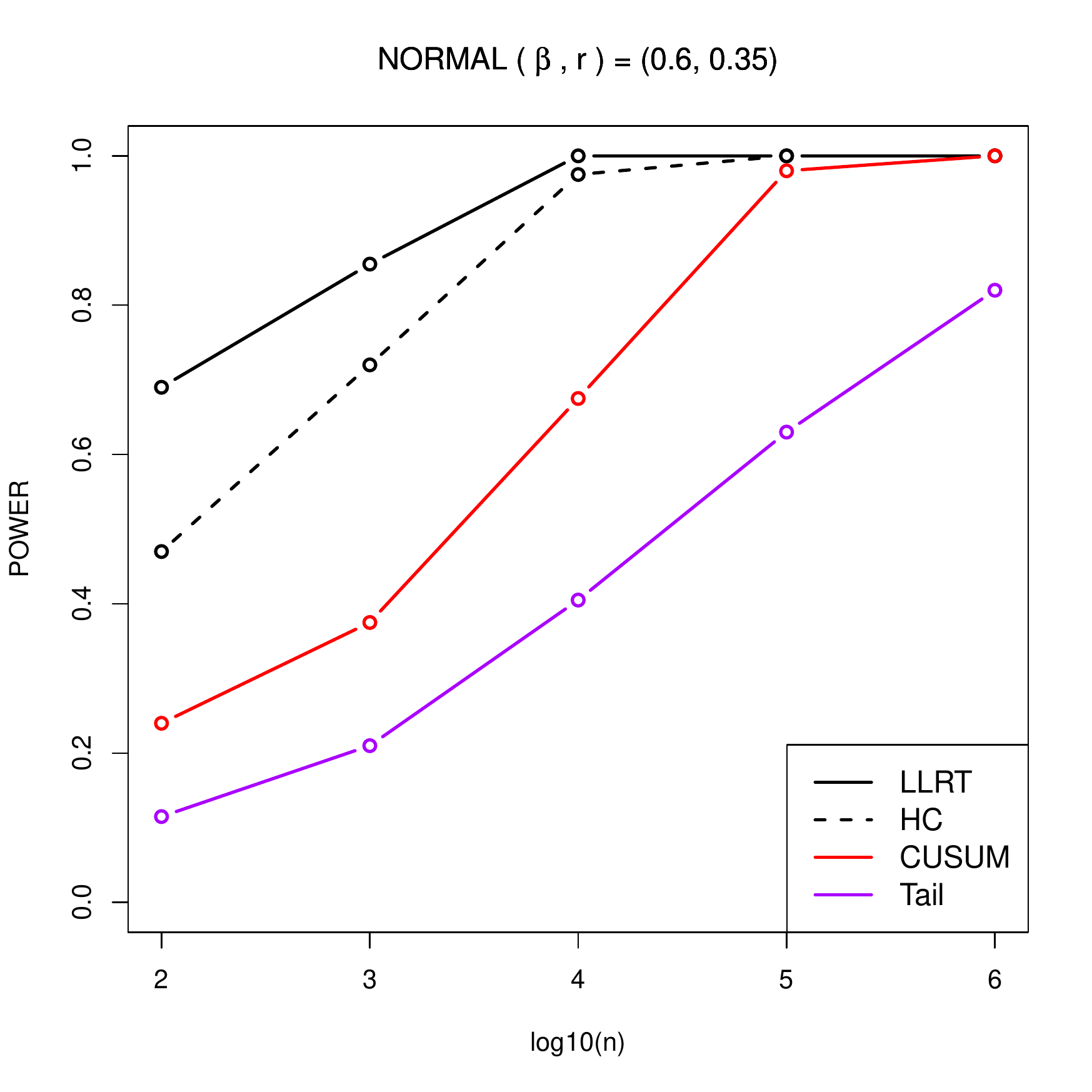}%
\includegraphics[scale=.3]{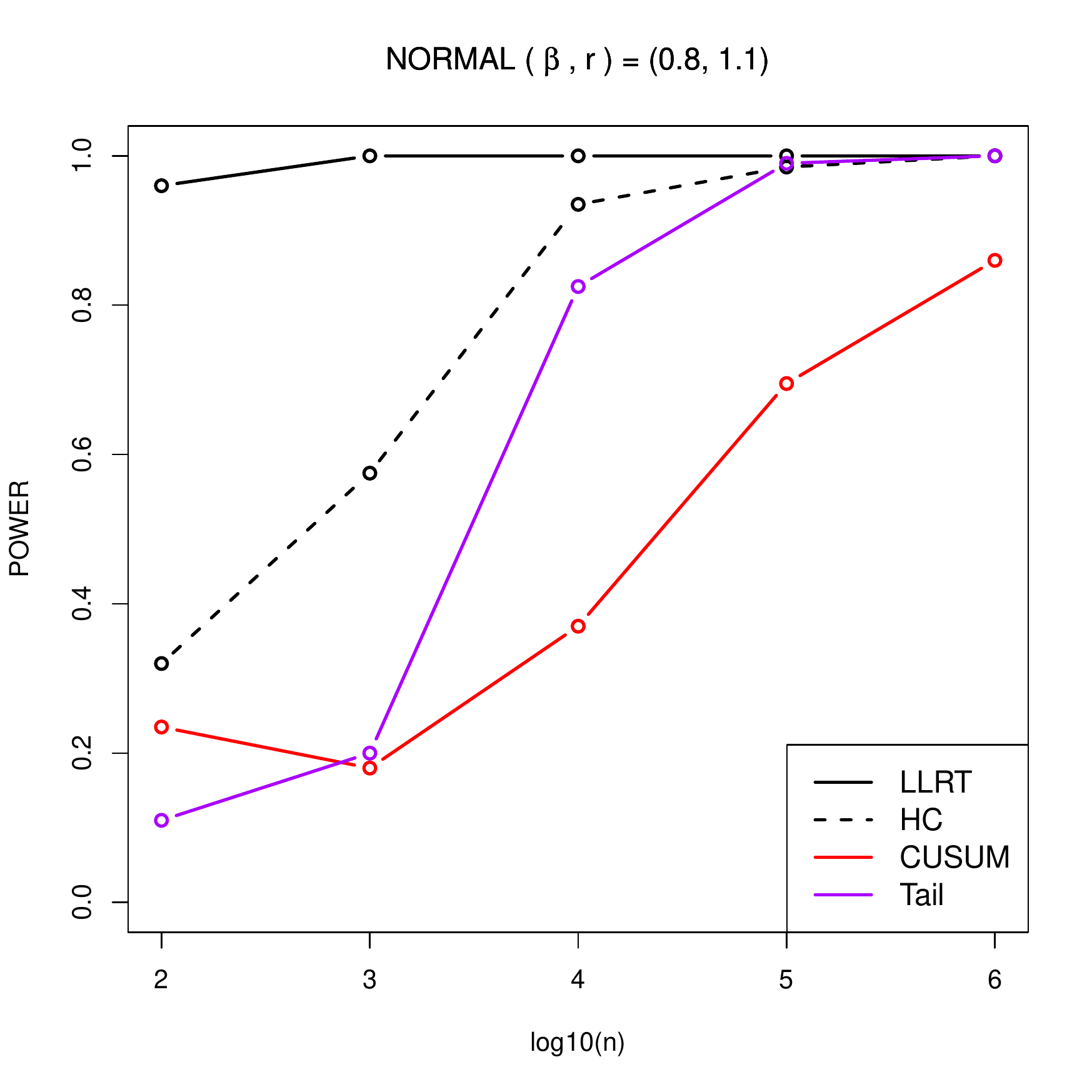}
\caption{Simulation results for the normal mixture model in three specific regimes with varying sample size.}
\label{fig:normal2}
\end{figure}

{\em Dense regime.}
We fixed $(\beta, s) = (0.2, 0.35)$, and chose $n = 10^2, 10^3, 10^4, 10^5$, $10^6$ as sample sizes, with number of positives 40, 251, 1585, $10^4$, 63096 respectively. 
We see that, for all test, the power increases rapidly with the sample size. 

{\em Sparse regimes.} For the moderately sparse regime, we fixed $(\beta, r) = (0.6, 0.35)$, and chose $n = 10^2, 10^3, 10^4, 10^5$, $10^6$, with number of positives $16, 30, 40, 100, 251$, respectively.
For the very sparse regime, we fixed $(\beta, r) = (0.8, 1.1)$, and chose $n = 10^2, 10^3, 10^4, 10^5$, $10^6$, with number of positives $3, 4, 6, 10, 16$, respectively. 
In both cases, the CUSUM sign and the tail-run tests are more affected by the small sample sizes than the LRT or HC.

\section{Discussion}
\label{sec:discussion}

%In the mixture model \eqref{h1}, the higher criticism approach of \cite{dj04} was shown to be near-optimal when $F$ is known to the statistician.  See \citep{cai-12} for fairly general results in the sparse regime.
%In the situation where $F$ is symmetric, we addressed the same testing problem and studied the performance of various distribution-free tests that do not require knowledge of $F$.
%The classical tests for the median --- the $t$-test and sign test --- were shown to be near-optimal in the dense regime $\sqrt{n} \eps_n \to 0$ and to have little or no power in the sparse regime $\sqrt{n} \eps_n \to 0$.
%The signed-rank and Smirnov tests behave similarly.
%The total-number-of-runs test is less powerful, being suboptimal in the dense regime.
%The situation is reversed for the longest-run test, which has little power in the dense regime, but has some power under the sparse regime, and in fact, is near-optimal in the very sparse regime for a generalized Gaussian mixture model with same parameter $\gamma \ge 1$.
%We introduced two new tests.  
%The tail-run test is based on the length of first run and its behavior is similar to that of the longest-run test in a generalized Gaussian mixture model with same parameter $\gamma \ge 1$, although in finite-sample simulations it shows an advantage when $\gamma > 1$.
%When $\gamma < 1$, the longest-run test is strictly more powerful.
%The second test we proposed is the cumulative sum sign test, a normalized version of the Smirnov test.
%It is near-optimal in all models and in all regimes studied here, and performs quite well numerically.

\subsection{Beyond the generalized Gaussian mixture model}
Although we used the generalized Gaussian mixture model as a benchmark for gaging the performance of the various tests studied here, this can be done in much more generality.  Assume $F=G$ for simplicity.
\bitem
\item For the dense regime, if $F$ differentiable and satisfies the conditions of \prpref{lower}, then all tests are asymptotically powerless when $\sqrt{n} \eps_n \mu_n \to 0$.  On the other hand, if $F$ is differentiable at 0 with $F'(0) > 0$, then the CUSUM sign, $t$-, sign, signed-rank and Smirnov tests are all asymptotically powerful when $\sqrt{n} \eps_n \mu_n \to \infty$. 
\item For the sparse regime, all the results apply in the same way if instead of a strict generalized Gaussian distribution \eqref{gg-f} we have $f(x) = f(-x)$ and 
\[
\lim_{x \to \infty} \frac1{x^\gamma} \log f(x) = -\frac1\gamma.
\]
In particular, the CUSUM sign test achieves the detection boundary in all these models, simultaneously.
\eitem

\subsection{Positive and negative effects}
A crucial assumption is that all the effects have same sign (here assumed positive).
When the effects can be negative or positive in the same experiment, then the problem is very different, as the assumption that $F$ is symmetric does not really help, since now the contamination can also be symmetric.
This is for instance the case in the canonical model:
\[
X_1, \dots, X_n \ \iid \ (1-\eps_n) \cN(0,1) + \frac{\eps_n}2 \cN(- \mu_n,1) + \frac{\eps_n}2 \cN(\mu_n,1).
\]
It is known that the detection boundary remains the same for generalized Gaussian mixture models in the sparse regime, and that the higher criticism remains near-optimal.
However, we do not know how to design a near-optimal distribution-free test in such a situation.  
Perhaps there is no distribution-free test that matches the performance of the higher criticism.

\section*{Acknowledgements}

We would like to thank Jason Schweinsberg for helpful discussions.
%The comments of two anonymous referees and an associate editor helped improve the presentation tremendously.
This work was partially supported by a grant from the Office of Naval Research (N00014-09-1-0258).

\bibliographystyle{chicago}
\bibliography{sparse-free.bib}

\begin{thebibliography}{}

\bibitem[\protect\citeauthoryear{Ahmad and Li}{Ahmad and Li}{1997}]{MR1443358}
Ahmad, I.~A. and Q.~Li (1997).
\newblock Testing symmetry of an unknown density function by kernel method.
\newblock {\em J. Nonparametr. Statist.\/}~{\em 7\/}(3), 279--293.

\bibitem[\protect\citeauthoryear{Aki}{Aki}{1987}]{MR930523}
Aki, S. (1987).
\newblock On nonparametric tests for symmetry.
\newblock {\em Ann. Inst. Statist. Math.\/}~{\em 39\/}(3), 457--472.

\bibitem[\protect\citeauthoryear{Anderson and Darling}{Anderson and
  Darling}{1952}]{MR0050238}
Anderson, T.~W. and D.~A. Darling (1952).
\newblock Asymptotic theory of certain ``goodness of fit'' criteria based on
  stochastic processes.
\newblock {\em Ann. Math. Statistics\/}~{\em 23}, 193--212.

\bibitem[\protect\citeauthoryear{Antille, Kersting, and Zucchini}{Antille
  et~al.}{1982}]{MR675891}
Antille, A., G.~Kersting, and W.~Zucchini (1982).
\newblock Testing symmetry.
\newblock {\em J. Amer. Statist. Assoc.\/}~{\em 77\/}(379), 639--646.

\bibitem[\protect\citeauthoryear{Arcones and Gin{\'e}}{Arcones and
  Gin{\'e}}{1991}]{MR1126334}
Arcones, M.~A. and E.~Gin{\'e} (1991).
\newblock Some bootstrap tests of symmetry for univariate continuous
  distributions.
\newblock {\em Ann. Statist.\/}~{\em 19\/}(3), 1496--1511.

\bibitem[\protect\citeauthoryear{Arias-Castro, Cand\`es, and Plan}{Arias-Castro
  et~al.}{2011}]{anova-hc}
Arias-Castro, E., E.~J. Cand\`es, and Y.~Plan (2011).
\newblock Global testing under sparse alternatives: Anova, multiple comparisons
  and the higher criticism.
\newblock {\em Ann. Statist.\/}.
\newblock To appear.

\bibitem[\protect\citeauthoryear{Arratia, Goldstein, and Gordon}{Arratia
  et~al.}{1989}]{MR972770}
Arratia, R., L.~Goldstein, and L.~Gordon (1989).
\newblock Two moments suffice for {P}oisson approximations: the {C}hen-{S}tein
  method.
\newblock {\em Ann. Probab.\/}~{\em 17\/}(1), 9--25.

\bibitem[\protect\citeauthoryear{Baklizi}{Baklizi}{2007}]{MR2413575}
Baklizi, A. (2007).
\newblock Testing symmetry using a trimmed longest run statistic.
\newblock {\em Aust. N. Z. J. Stat.\/}~{\em 49\/}(4), 339--347.

\bibitem[\protect\citeauthoryear{Berk}{Berk}{1973}]{MR0350815}
Berk, K.~N. (1973).
\newblock A central limit theorem for {$m$}-dependent random variables with
  unbounded {$m$}.
\newblock {\em Ann. Probability\/}~{\em 1}, 352--354.

\bibitem[\protect\citeauthoryear{Caba{\~n}a and Caba{\~n}a}{Caba{\~n}a and
  Caba{\~n}a}{2000}]{MR1821437}
Caba{\~n}a, A. and E.~M. Caba{\~n}a (2000).
\newblock Tests of symmetry based on transformed empirical processes.
\newblock {\em Canad. J. Statist.\/}~{\em 28\/}(4), 829--839.

\bibitem[\protect\citeauthoryear{Cai, Jeng, and Jin}{Cai
  et~al.}{2011}]{cai2010optimal}
Cai, T.~T., X.~J. Jeng, and J.~Jin (2011).
\newblock Optimal detection of heterogeneous and heteroscedastic mixtures.
\newblock {\em J. R. Stat. Soc. Ser. B Stat. Methodol.\/}~{\em 73\/}(5),
  629--662.

\bibitem[\protect\citeauthoryear{Cai and Wu}{Cai and Wu}{2012}]{cai-12}
Cai, T.~T. and Y.~Wu (2012).
\newblock Optimal detection for sparse mixtures.
\newblock Available online \url{http://arxiv.org/abs/1211.2265}.

\bibitem[\protect\citeauthoryear{Cohen and Menjoge}{Cohen and
  Menjoge}{1988}]{MR926417}
Cohen, J.~P. and S.~S. Menjoge (1988).
\newblock One-sample run tests of symmetry.
\newblock {\em J. Statist. Plann. Inference\/}~{\em 18\/}(1), 93--100.

\bibitem[\protect\citeauthoryear{Darling and Erd{\"o}s}{Darling and
  Erd{\"o}s}{1956}]{MR0074712}
Darling, D.~A. and P.~Erd{\"o}s (1956).
\newblock A limit theorem for the maximum of normalized sums of independent
  random variables.
\newblock {\em Duke Math. J.\/}~{\em 23}, 143--155.

\bibitem[\protect\citeauthoryear{Delaigle, Hall, and Jin}{Delaigle
  et~al.}{2011}]{MR2815777}
Delaigle, A., P.~Hall, and J.~Jin (2011).
\newblock Robustness and accuracy of methods for high dimensional data analysis
  based on {S}tudent's {$t$}-statistic.
\newblock {\em J. R. Stat. Soc. Ser. B Stat. Methodol.\/}~{\em 73\/}(3),
  283--301.

\bibitem[\protect\citeauthoryear{Donoho and Jin}{Donoho and Jin}{2004}]{dj04}
Donoho, D. and J.~Jin (2004).
\newblock Higher criticism for detecting sparse heterogeneous mixtures.
\newblock {\em Ann. Statist.\/}~{\em 32\/}(3), 962--994.

\bibitem[\protect\citeauthoryear{Einmahl and McKeague}{Einmahl and
  McKeague}{2003}]{MR1997030}
Einmahl, J. H.~J. and I.~W. McKeague (2003).
\newblock Empirical likelihood based hypothesis testing.
\newblock {\em Bernoulli\/}~{\em 9\/}(2), 267--290.

\bibitem[\protect\citeauthoryear{Erd{\H{o}}s and R{\'e}nyi}{Erd{\H{o}}s and
  R{\'e}nyi}{1970}]{MR0272026}
Erd{\H{o}}s, P. and A.~R{\'e}nyi (1970).
\newblock On a new law of large numbers.
\newblock {\em J. Analyse Math.\/}~{\em 23}, 103--111.

\bibitem[\protect\citeauthoryear{Hall and Jin}{Hall and Jin}{2008}]{hj08}
Hall, P. and J.~Jin (2008).
\newblock Properties of higher criticism under strong dependence.
\newblock {\em Ann. Statist.\/}~{\em 36\/}(1), 381--402.

\bibitem[\protect\citeauthoryear{Hall and Jin}{Hall and Jin}{2010}]{hj09}
Hall, P. and J.~Jin (2010).
\newblock Innovated higher criticism for detecting sparse signals in correlated
  noise.
\newblock {\em Ann. Statist.\/}~{\em 38\/}(3), 1686--1732.

\bibitem[\protect\citeauthoryear{Hettmansperger}{Hettmansperger}{1984}]{MR758442}
Hettmansperger, T.~P. (1984).
\newblock {\em Statistical inference based on ranks}.
\newblock Wiley Series in Probability and Mathematical Statistics: Probability
  and Mathematical Statistics. New York: John Wiley \& Sons, Inc.

\bibitem[\protect\citeauthoryear{Ingster, Tsybakov, and Verzelen}{Ingster
  et~al.}{2010}]{ingster2010detection}
Ingster, Y., A.~Tsybakov, and N.~Verzelen (2010).
\newblock Detection boundary in sparse regression.
\newblock {\em Electronic Journal of Statistics\/}~{\em 4}, 1476--1526.

\bibitem[\protect\citeauthoryear{Ingster}{Ingster}{1997}]{MR1456646}
Ingster, Y.~I. (1997).
\newblock Some problems of hypothesis testing leading to infinitely divisible
  distributions.
\newblock {\em Math. Methods Statist.\/}~{\em 6\/}(1), 47--69.

\bibitem[\protect\citeauthoryear{Ingster}{Ingster}{2002a}]{ingster2002a}
Ingster, Y.~I. (2002a).
\newblock Adaptive detection of a signal of growing dimension {\rm i}.
\newblock {\em Math. Methods Statist.\/}~{\em 10}, 395--421.

\bibitem[\protect\citeauthoryear{Ingster}{Ingster}{2002b}]{ingster2002b}
Ingster, Y.~I. (2002b).
\newblock Adaptive detection of a signal of growing dimension {\rm ii}.
\newblock {\em Math. Methods Statist.\/}~{\em 11}, 37--68.

\bibitem[\protect\citeauthoryear{Jager and Wellner}{Jager and
  Wellner}{2007}]{jager}
Jager, L. and J.~Wellner (2007).
\newblock Goodness-of-fit tests via phi-divergences.
\newblock {\em Ann. Statist.\/}~{\em 35\/}(5), 2018---2053.

\bibitem[\protect\citeauthoryear{Jennen-Steinmetz and Gasser}{Jennen-Steinmetz
  and Gasser}{1986}]{MR886465}
Jennen-Steinmetz, C. and T.~Gasser (1986).
\newblock The asymptotic power of runs tests.
\newblock {\em Scand. J. Statist.\/}~{\em 13\/}(4), 263--269.

\bibitem[\protect\citeauthoryear{Jin}{Jin}{2003}]{JinPhD}
Jin, J. (2003).
\newblock {\em Detecting and Estimating Sparse Mixtures}.
\newblock Ph.\ D. thesis, Stanford University.

\bibitem[\protect\citeauthoryear{Lehmann and Romano}{Lehmann and
  Romano}{2005}]{TSH}
Lehmann, E.~L. and J.~P. Romano (2005).
\newblock {\em Testing statistical hypotheses\/} (Third ed.).
\newblock Springer Texts in Statistics. New York: Springer.

\bibitem[\protect\citeauthoryear{McWilliams}{McWilliams}{1990}]{McW}
McWilliams, T.~P. (1990).
\newblock A distribution-free test for symmetry based on a runs statistic.
\newblock {\em J. Amer. Statist. Assoc.\/}~{\em 85\/}(412), 1130--1133.

\bibitem[\protect\citeauthoryear{Mira}{Mira}{1999}]{MR1731875}
Mira, A. (1999).
\newblock Distribution-free test for symmetry based on {B}onferroni's measure.
\newblock {\em J. Appl. Statist.\/}~{\em 26\/}(8), 959--972.

\bibitem[\protect\citeauthoryear{Mosteller}{Mosteller}{1941}]{MR0004453}
Mosteller, F. (1941).
\newblock Note on an application of runs to quality control charts.
\newblock {\em Ann. Math. Statistics\/}~{\em 12}, 228--232.

\bibitem[\protect\citeauthoryear{Orlov}{Orlov}{1972}]{MR0301840}
Orlov, A.~I. (1972).
\newblock Testing the symmetry of a distribution.
\newblock {\em Teor. Verojatnost. i Primemen.\/}~{\em 17}, 372--377.

\bibitem[\protect\citeauthoryear{Osmoukhina}{Osmoukhina}{2001}]{MR1861381}
Osmoukhina, A.~V. (2001).
\newblock Large deviations probabilities for a test of symmetry based on kernel
  density estimator.
\newblock {\em Statist. Probab. Lett.\/}~{\em 54\/}(4), 363--371.

\bibitem[\protect\citeauthoryear{Rothman and Woodroofe}{Rothman and
  Woodroofe}{1972}]{MR0365885}
Rothman, E.~D. and M.~Woodroofe (1972).
\newblock A {C}ram\'er-von {M}ises type statistic for testing symmetry.
\newblock {\em Ann. Math. Statist.\/}~{\em 43}, 2035--2038.

\bibitem[\protect\citeauthoryear{Schuster and Barker}{Schuster and
  Barker}{1987}]{MR883122}
Schuster, E.~F. and R.~C. Barker (1987).
\newblock Using the bootstrap in testing symmetry versus asymmetry.
\newblock {\em Comm. Statist. B---Simulation Comput.\/}~{\em 16\/}(1), 69--84.

\bibitem[\protect\citeauthoryear{Smirnov}{Smirnov}{1947}]{MR0021260}
Smirnov, N.~V. (1947).
\newblock Sur un crit\`ere de sym\'etrie de la loi de distribution d'une
  variable al\'eatoire.
\newblock {\em C. R. (Doklady) Acad. Sci. URSS (N.S.)\/}~{\em 56}, 11--14.

\bibitem[\protect\citeauthoryear{Wald and Wolfowitz}{Wald and
  Wolfowitz}{1940}]{MR0002083}
Wald, A. and J.~Wolfowitz (1940).
\newblock On a test whether two samples are from the same population.
\newblock {\em Ann. Math. Statistics\/}~{\em 11}, 147--162.

\bibitem[\protect\citeauthoryear{Wilcoxon}{Wilcoxon}{1945}]{wilcoxon1945individual}
Wilcoxon, F. (1945).
\newblock Individual comparisons by ranking methods.
\newblock {\em Biometrics Bulletin\/}~{\em 1\/}(6), 80--83.

\end{thebibliography}

\setcounter{section}{0}
\setcounter{equation}{0}
\setcounter{thm}{0}
\setcounter{prp}{0}
\setcounter{lem}{0}
\setcounter{cor}{0}

\numberwithin{equation}{section}
\renewcommand{\thethm}{\Alph{section}.\arabic{thm}}
\renewcommand{\theprp}{\Alph{section}.\arabic{prp}}
\renewcommand{\thelem}{\Alph{section}.\arabic{lem}}
\renewcommand{\thecor}{\Alph{section}.\arabic{cor}}

\bigskip
\appendix

\section{Lower bounds} \label{sec:lower} 

We let $\P_0, \E_0, \Var_0$ and $\P_1, \E_1, \Var_1$ denote the probability, expectation and variance under the null and alternative, respectively.

\subsection{Truncated moment method}

\cite{MR1456646} devised a general method for showing that the LRT (and therefore any other test) is asymptotically powerless.  It is based on the first two moments of a truncated likelihood ratio.  It yields the following.

\begin{lem} \label{lem:trunc}
Let $f$ and $g$ denote the densities of $F$ and $G$ with respect to a dominating measure.  Then the hypotheses merge asymptotically when there is a sequence $(x_n)$ such that 
\beq \label{trunc1}
n \bar F(x_n) \to 0, \qquad 
n \eps_n \bar G(x_n -\mu_n) \to 0,
\eeq
and
\beq \label{trunc2}
n \eps_n^2 \left[ \int_{-\infty}^{x_n} \frac{g(x - \mu_n)^2}{f(x)} {\rm d}x - 1\right]_+ \to 0.
\eeq
\end{lem}

\begin{proof}
The likelihood ratio is given by
\[
L = \prod_{i=1}^n L_i, \quad L_i := 1-\eps_n + \eps_n \frac{g(X_i-\mu_n)}{f(X_i)}.
\]
The test $\{L > 1\}$ minimizes the risk at 
\[R^* := 1 - \frac12 \E_0|L-1| = 1 - \E_0 (1 - L)_+,\]
where $x_+ := \max(0, x)$ for any $x \in \bbR$.  (Note that $1-R^*$ is the total variation distance between $F$ and $G$.)
We do not work with $L$ directly, but truncate it first.  Define  
\[\tilde{L} = \prod_{i=1}^n L_i \cdot \mathbbm{1}_{A_i}, \quad A_i := \{X_i \leq x_n\}.\]
Using the fact that $\tilde L \le L$ and then the Cauchy-Schwarz inequality, we have
\[R^* \ge 1 - \E_0 (1 - \tilde L)_+ \ge 1 - \sqrt{\E_0 (1 - \tilde L)^2}.\]
And since 
\[\E_0 (1 - \tilde L)^2 = \big[\E_0 \tilde L^2 - 1\big] + 2 \big[1 - \E_0 \tilde L \big],\]
to prove that $R^* \to 1$, it suffices that
\[\E_0 \tilde L \ge 1 + o(1)  \quad \text{ and } \quad \E_0 \tilde L^2 \le 1 + o(1) .\]

For the first moment, we have
\[
\E_0 \tilde L = \prod_{i=1}^n \E_0 (\mathbbm{1}_{A_i} L_i) = \prod_{i=1}^n \E_1 \mathbbm{1}_{A_i} = \prod_{i=1}^n \P_1(A_i) = \P_1(X \le x_n)^n,
\]
with
\[
\P_1(X \le x_n)^n = \big[ (1-\eps_n)F(x_n) + \eps_n G(x_n -\mu_n) \big]^n \to 1,
\]
under \eqref{trunc1}.
Indeed, we use the fact that, for any sequence $(a_n)$ of positive reals, 
\beq \label{an}
a_n^n = \exp(n \log a_n) \to 1 \ \Leftrightarrow \ n \log a_n \sim n(a_n - 1) \to 0,
\eeq 
applying this with $a_n = (1-\eps_n)F(x_n) + \eps_n G(x_n -\mu_n)$, to get
\[
0 \le n(1 - a_n) \le n \bar F(x_n) + \eps_n \bar G(x_n -\mu_n) \to 0, 
\]
by \eqref{trunc1}.

For the second moment,
\[\E_0 \tilde{L}^2 = \prod_{i=1}^n \E_0 (\mathbbm{1}_{A_i} L_i^2) = \left[ \E_0 (\mathbbm{1}_{A_i} L_i^2) \right]^n,\]
with
\begin{align*}
\E_0 (\mathbbm{1}_{A_i} L_i^2) 
&= \int_{-\infty}^{x_n} \big(1-\eps_n + \eps_n \frac{g(x-\mu_n)}{f(x)}\big)^2 f(x) {\rm d}x \\
&\le 1 + a_n, \quad a_n := \eps_n^2 \left[ \int_{-\infty}^{x_n} \frac{g(x-\mu_n)^2}{f(x)} {\rm d}x - 1\right],
\end{align*}
using the fact that both $f$ and $g$ are probability density functions.
We then apply \eqref{an} and use \eqref{trunc2}.  
\end{proof}

\subsection{Proof of \prpref{gg}}
We may assume that $s < 1/2$.  We have
\[
\int \frac{f(x - \mu_n)^2}{f(x)} {\rm d}x - 1 = \int \big[\exp(h_n(x)) -1\big]^2 f(x) {\rm d}x,
\]
where $h_n(x) := \frac1\gamma (|x|^\gamma - |x-\mu_n|^\gamma)$.
When $|x| \le 2 \mu_n$, $h_n(x) = O(\mu_n^\gamma)$, so that 
\[
\int_{-2\mu_n}^{2\mu_n} \big[\exp(h_n(x)) -1\big]^2 f(x) {\rm d}x = O(\mu_n^{1+2\gamma}).
\]
When $|x| > 2 \mu_n$, we have 
\[
|h_n(x)| = \frac{|x|^\gamma}\gamma [1 - (1-\mu_n/|x|)^\gamma] \le \frac{|x|^\gamma}\gamma \cdot 2 \gamma \mu_n/|x| = 2 \mu_n |x|^{\gamma-1}.
\]
Hence, 
\begin{align*}
\int_{|x| > 2\mu_n} \big[\exp(h_n(x)) -1\big]^2 f(x) {\rm d}x 
&\le \int_{|x| > 2\mu_n} \big[\exp(2 \mu_n |x|^{\gamma-1}) -1\big]^2 f(x) {\rm d}x \\
&\asymp \mu_n^2 \int_{|x| > 2\mu_n} |x|^{2 \gamma-2} f(x) {\rm d}x \\
&\asymp a_n := \begin{cases}
\mu_n^2, & \text{if } \gamma > 1/2 \\
\mu_n^2 \log(1/\mu_n), & \text{if } \gamma = 1/2 \\
\mu_n^{1+2\gamma}, & \text{if } \gamma < 1/2.
\end{cases}
\end{align*}
We used dominated convergence in the last line.
Hence, by \lemref{trunc} (with $x_n = \infty$), the hypotheses merge asymptotically when $b_n := n \eps_n^2 (\mu_n^{1+2\gamma} \vee a_n) \to \infty$.  When $\gamma > 1/2$, $b_n = n \eps_n^2 \mu_n^2 = n^{s-\beta} \to 0$ when $s < \beta$.  When $\gamma=1/2$, $b_n = n \eps_n^2 \mu_n^2 \log(1/\mu_n) \asymp n^{s-\beta} \log n \to 0$ when $s < \beta$.  When $\gamma < 1/2$, $b_n = n \eps_n^2 \mu_n^{1+2\gamma} = n^{1-\beta-(1+2\gamma)(s-1/2)} \to 0$ when $s < \frac12 - \frac{1 - 2\beta}{1+ 2\gamma}$. 

We now show that the hypotheses separate completely when $\gamma \ge 1/2$ and $s > \beta$, or when $\gamma < 1/2$ and $s > \frac12 - \frac{1 - 2\beta}{1+ 2\gamma}$. We will show later that several tests (CUSUM, t, sign, signed-rank, Smirnov) are asymptotically powerful in this setting in the former situation, so we focus on the latter.
For this, it suffices to do as \cite{cai-12}, and show that 
\[
n \cH^2\big(f, (1-\eps_n) f + \eps_n f(\cdot-\mu_n)\big) \to \infty,
\]
where $\cH$ denote the Hellinger distance.  
When $\mu_n \le x \le 2\mu_n$, we have $h_n(x) \ge a_n := (2^\gamma-1) \mu_n^\gamma/\gamma.$
Hence,
\begin{align*}
& \cH^2\big(f, (1-\eps_n) f + \eps_n f(\cdot-\mu_n)\big) \\
&= \int \left[\sqrt{1 + \eps_n [\exp(h_n(x)) -1]} - 1\right]^2 f(x) {\rm d}x \\
&\ge \int_{\mu_n}^{2\mu_n} \left[\sqrt{1 + \eps_n a_n/\gamma} - 1\right]^2 f(x) {\rm d}x 
\asymp \eps_n^2 a_n^2 \mu_n \asymp \eps_n^2 \mu_n^{1 + 2 \gamma}.
\end{align*}
The result comes from that.

\subsection{A general information bound for the dense regime} \label{sec:lb}

The following result does not require symmetry.  Note that $(F,G,\eps_n,\mu_n)$ below are implicitly known.

\begin{prp} \label{prp:lower}
Assume that $\sqrt{n} \eps_n \to \infty$.  When $F \ne G$, then there is a test that asymptotically separates $H_0^n$ and $H_1^n$.  When $F=G$, assume that $F$ is symmetric about 0 and has a differentiable density $f$ that satisfies $f(x-\mu) = f(x) - \mu f'(x) + \mu^2 h(x,\mu)$ for all $\mu \ge 0$ and all $x \in \bbR$, with    
\[
\int \frac{f'(x)^2}{f(x)} {\rm d}x <\infty, \quad \sup_{\mu\ge0} \int \frac{h(x,\mu)^2}{f(x)} {\rm d}x < \infty.
\]
Then the hypotheses are asymptotically inseparable if $\sqrt{n} \eps_n \mu_n \to 0.$
\end{prp}

Compare with the performance bounds obtained for the CUSUM sign test, the $t$-test, the sign test, signed-rank test, and the Smirnov test, which were shown to be asymptotically powerful when $\sqrt{n} \eps_n \mu_n \to \infty$ under mild additional conditions. 
Note that \prpref{lower} is strong enough to imply \prpref{gg} when $\gamma \ge 2$.  

\begin{proof}
First assume that $F \ne G$.  Extracting a subsequence if needed, we may assume that $\mu_n \to \mu$ for some $\mu \in [0, \infty]$.  If $\mu = \infty$, then consider the test that rejects when $Q := \# \{i:X_i \ge \mu_n\}$ is too large.  
%Under the null, $Q \sim \Bin(n, \bar F(\mu_n))$.  Under the alternative, $Q \sim \Bin(n, (1-\eps_n) \bar F(\mu_n) + \eps_n/2)$, using the fact that $G$ has zero median.
We have
\[
\E_0 Q = n \bar F(\mu_n), \quad \Var_0(Q) \le n/4,
\]
and using the fact that $G$ has zero median,
\[
\E_1 Q = n [(1-\eps_n) \bar F(\mu_n) + \eps_n/2], \quad \Var_1(Q) \le n/4.
\]
Therefore,
\[
\frac{\E_1 Q - \E_0 Q}{\sqrt{\Var_0(Q) \vee \Var_1(Q)}} \ge \sqrt{n} \eps_n [1 - 2\bar F(\mu_n)] \to \infty,
\]
and we conclude with \lemref{basic} that there is a test based on $Q$ that is asymptotically powerful.  If $\mu < \infty$, let $A$ be a measurable subset of $\bbR$ such that $F(A) < G(A-\mu)$.
This is possible since $F \ne G(\cdot -\mu)$.  (If $\mu \ne 0$, this comes from the fact that ${\rm med}(F) = 0$ while ${\rm med}(G(\cdot -\mu)) = \mu$.)  We then consider the test based on $Q := \# \{i:X_i\in A\}$ and reason as above.  We have 
\[\E_1 Q - \E_0 Q = n \eps_n [G(A-\mu_n) - F(A)] \sim n \eps_n [G(A - \mu) - F(A)],\]
and $\Var_0(Q) \vee \Var_1(Q) \le n/4$, so that 
\[
\frac{\E_1 Q - \E_0 Q}{\sqrt{\Var_0(Q) \vee \Var_1(Q)}} \ge (1 + o(1)) \frac{n \eps_n [G(A - \mu) - F(A)]}{\sqrt{n/4}} \asymp \sqrt{n} \eps_n \to \infty.
\]

We now assume that $F=G$.  We first note that $f'$ is integrable since, by the Cauchy-Schwarz inequality,
\[
\int |f'(x)| {\rm d}x \le \sqrt{\int \frac{f'(x)^2}{f(x)} {\rm d}x \cdot \int f(x) {\rm d}x} < \infty.
\]
Then, because $f$ is even, $f'$ is odd, and therefore $\int f'(x) {\rm d}x = 0$.
We have
\begin{align*}
\int \frac{f(x-\mu_n)^2}{f(x)} {\rm d}x - 1 
&= \int \frac{\big[f(x) -\mu_n f'(x) + \mu_n^2 h(x,\mu_n) \big]^2}{f(x)} {\rm d}x - 1 \\
&= \mu_n^2 \int \frac{f'(x)^2}{f(x)} {\rm d}x + 2 \mu_n^2 \int h(x,\mu_n) {\rm d}x \\
& \quad + \mu_n^4 \int \frac{h(x,\mu_n)^2}{f(x)} {\rm d}x - 2 \mu_n^3 \int \frac{f'(x) h(x,\mu_n)}{f(x)} {\rm d}x,
\end{align*}
with, by the Cauchy-Schwarz inequality, 
\[
\sup_n \left| \int h(x,\mu_n) {\rm d}x \right| \le \sqrt{\sup_{\mu \ge 0} \int \frac{h(x,\mu)^2}{f(x)} {\rm d}x} < \infty,
\]
and
\[
\sup_n \left|\int \frac{f'(x) h(x,\mu_n)}{f(x)} {\rm d}x\right| \le \sqrt{\int \frac{f'(x)^2}{f(x)} {\rm d}x \cdot \sup_{\mu \ge 0} \int \frac{h(x,\mu)^2}{f(x)} {\rm d}x} < \infty.
\]
Hence,
\[
n \eps_n^2 \left[\int \frac{f(x-\mu_n)^2}{f(x)} {\rm d}x - 1\right]^2 = O(n \eps_n^2 \mu_n^2),
\]
and we conclude with \lemref{trunc} (with $x_n = \infty$). 
\end{proof}

\subsection{Generalized Gaussian mixture model (different parameters)} 
\label{sec:expo2}

Suppose $F$ and $G$ are generalized Gaussian with parameters $\gamma \ne \eta$.  By \prpref{lower}, in the dense regime the two hypotheses $H_0^n$ and $H_1^n$ are asymptotically separable, so we focus on the sparse regime where $\eps_n = n^{-\beta}$ with $1/2 < \beta < 1$.  

{\em Case $\gamma > \eta$.}
Here, $g$ has heavy tails compared to $f$, so much so that the max test --- which rejects for large values of $\max_i X_i$ --- is asymptotically powerful as soon as $\beta < 1$, even if $\mu_n = 0$.  
Indeed, with high probability under the null, $\max_i X_i \le 2 (\gamma \log n)^{1/\gamma} \asymp (\log n)^{1/\gamma}$, while under the alternative (with $\mu_n=0$), at least $n \eps_n/2$ points are sampled from $G$, and the maximum of them exceeds $\frac12 (\eta \log (n \eps_n/2))^{1/\eta} \asymp (\log n)^{1/\eta}$.

{\em Case $\gamma < \eta$.}
Here, $g$ has lighter tails than $f$, and as a consequence, the max test has very little power.  This situation is more interesting.  Following standard lines, we obtain the following.  

\begin{prp} \label{prp:gg-diff}
Suppose $F$ and $G$ are generalized Gaussian with parameters $\gamma < \eta$, and that we parameterize $\eps_n$ and $\mu_n$ as in \eqref{eps} and \eqref{mu-gamma}.  Then the hypotheses merge asymptotically when $r < 2 \beta - 1$. 
\end{prp}
This coincides with the detection boundary when $F=G$ is generalized Gaussian with exponent $\gamma \le 1$.  We note that the result is sharp.  For example, the CUSUM sign test achieves this detection boundary.  (We invite the reader to verify this based on \prpref{cusum}.)

\begin{proof}
We want to apply \lemref{trunc}.  
The first condition in \eqref{trunc1} holds when $x_n \geq 2 (\gamma \log n)^{1/\gamma}$, while the second condition in \eqref{trunc1} is fulfilled when $x_n \geq \mu_n + 2 (\eta \log (n\eps_n))^{1/\eta} \sim \mu_n = (\gamma r \log n)^{1/\gamma}$.
Hence, the choice $x_n = 2 (\gamma \log n)^{1/\gamma}$ is valid.
%We will fix $\delta > 0$ small later on.

We now turn to \eqref{trunc2}, where $g(x - \mu_n)^2/f(x) \propto \exp(h_n(x))$, with $h_n(x) := -2 |x-\mu_n|^{\eta}/\eta + |x|^{\gamma}/\gamma$. 
For $x \le 0$, we have 
\[
\int_{-\infty}^0 \exp(h_n(x)) {\rm d}x < \int_{-\infty}^0 \exp\big(-2|x|^\eta/\eta + |x|^\gamma/\gamma\big) {\rm d}x < \infty,
\]
because $\gamma < \eta$.
We therefore focus $x > 0$.
%There, $h_n(x)$ is differentiable everywhere except at $x = \mu_n$.
We see that $h_n$ is increasing over $(0,z_n)$ and decreasing over $(z_n, \infty)$, where $z_n$ be the (unique) root of $h_n'(x) = 0$ over $(0, \infty)$, specifically, $z_n > \mu_n$ satisfies $z_n^{\gamma-1} = 2 (z_n -\mu_n)^{\eta-1}$.
Expressing $z_n$ as $z_n = (\gamma t_n \log n)^{1/\gamma}$ for some $t_n > r$, we have
\[
(\gamma t_n \log n)^{\frac{\gamma - 1}{\gamma}} = 2(\gamma \log n)^{\frac{\eta-1}{\gamma}} (t_n^{\frac1{\gamma}} - r^{\frac1{\gamma}})^{\eta -1}.
\]
Since $\eta > \gamma$, we necessarily have $t_n \to r$ and thus $z_n \sim \mu_n$. Hence, we have 
\[
h_n(z_n) = -2\frac{(z_n -\mu_n)^{\eta}}{\eta} + \frac{z_n^{\gamma}}{\gamma} \leq \frac{z_n^{\gamma}}{\gamma} \sim \frac{\mu_n^{\gamma}}{\gamma} = r \log n,
\]
leading to
\begin{align*}
n \eps_n^2 \cdot \int_0^{x_n} \exp(h_n(x)) {\rm d}x 
&\le n^{1-2\beta} \cdot x_n \cdot \exp(h_n(z_n)) \\
&\asymp (\log n)^{1/\gamma} n^{r +o(1) + 1-2\beta} \to 0,
\end{align*}
when $r < 2\beta - 1$.
\end{proof}

\section{Performance bounds}
\label{sec:perf}

We let $\P_1$, $\E_1$ and $\Var_1$ denote the probability, expectation and variance under the mixture model \eqref{h1}.  The corresponding notation for the null distribution \eqref{h0} --- corresponding to \eqref{h1} with $\eps_n = 0$ --- is $\P_0$, $\E_0$ and $\Var_0$.  

\subsection{Proof of \prpref{cusum}}

By \citep{MR0074712}, under the null, $M/\sqrt{2 \log\log(n)} \to_P 1$. 
Define $\omega_n = 2\sqrt{\log\log(n)}$, so that $\P_0(M \geq \omega_n) \to 0$, as $n \to \infty$.

First, assume that \eqref{prp_cusum1} holds.  Recall the definition of $S_n$ in \eqref{smv}.
Since $\E_1 S_n = n \eps_n (1 - 2G(-\mu_n))$ and $\Var_1(S_n) \le n$, we have $\E_1 S_n \gg \sqrt{\Var_1(S_n)}$ by \eqref{prp_cusum1}.
Hence, by Chebyshev's inequality, $S_n \ge \frac12 \E_1 S_n$ with probability tending to one, implying that 
\[M \ge \frac1{\sqrt{n}} S_n \ge \frac1{2 \sqrt{n}} \E_1 S_n \ge \frac12 \sqrt{n} \eps_n (1 - 2G(-\mu_n)) \gg \omega_n,\] 
by \eqref{prp_cusum1}.
Therefore the test is asymptotically powerful.

Finally, assume that \eqref{prp_cusum3} holds.
Let $N^+ = \# \{i: X_i > x_n\}$, $N^- = \# \{i: X_i < -x_n\}$ and $N = N^+ + N^-$.
We have $M \ge (N^+ - N^-)/\sqrt{N}$, with $N^\pm \sim \Bin(n, p^\pm)$ and $N \sim \Bin(n,p)$, where $p^+ := (1-\eps_n)\bar{F}(x_n) + \eps_n \bar{G}(x_n-\mu_n)$, $p^- := (1-\eps_n)F(-x_n) + \eps_n G(-x_n-\mu_n)$ and $p := p^+ + p^-$.  
Let $a_n = n (1-\eps_n) \bar{F}(x_n)$, $b^+_n = n \eps_n \bar{G}(x_n-\mu_n)$ and $b^-_n = n G(-x_n-\mu_n)$.
We have 
\[N \sim_P np = 2 a_n + b^+_n + b^-_n \to \infty,\]
since $\sqrt{b_n^+} \ge (b^+_n - b^-_n)/\sqrt{a_n + b^+_n + b^-_n} \to \infty$, where the divergence is due to \eqref{prp_cusum3}.
We also have
\[
\E_1 (N^+ - N^-) = n (p^+ - p^-) = b^+_n - b^-_n,
\]
and since $N^+ | N \sim \Bin(N, q)$, with $q := p^+/p$, by the law of total variance, 
\begin{align*}
\Var_1(N^+ - N^-) 
&= \Var_1\big(\E_1[2 N^+ - N|N]\big) + \E_1\big[\Var_1(2 N^+-N|N)\big] \\
&= \Var_1( (2 q - 1) N) + \E_1[4 N q(1-q)]  \\
&= (2q - 1)^2 n p(1-p) + 4 q(1-q) n p \\
&\le 2 n p = 4 a_n + 2 (b^+_n + b^-_n).
\end{align*}
Hence, by Chebyshev's inequality, $N^+ - N^- = b^+_n - b^-_n + O_P(\sqrt{a_n + b^+_n + b^-_n})$.
We therefore have
\begin{align*}
\frac{N^+ - N^-}{\sqrt{N}}
& = \frac{b^+_n - b^-_n + O_P(\sqrt{a_n + b^+_n + b^-_n})}{\sqrt{(1+o_P(1)) (2 a_n + b^+_n + b^-_n)}} \\
&= (1+o_P(1)) \frac{b^+_n - b^-_n}{\sqrt{a_n+b^+_n + b^-_n}} + O_P(1) \gg \omega_n,
\end{align*}
by \eqref{prp_cusum3}.

\subsection{Proof of \prpref{tail}}

We first show that the tail-run test is asymptotically powerful when \eqref{prp_tail1} holds.  
Since $\tail = O_P(1)$ under the null, it suffices to show that $\tail \to \infty$ under the alternative.
We first note that
\begin{align*}
\P(\max\{|X_i| : X_i < 0\} > x_n) 
&= \P(\min_i X_i < -x_n) \\
&\le n \big( (1-\eps_n)F(-x_n) + \eps_n G(-x_n-\mu_n)\big) \to 0,
\end{align*}
by the union bound, the first two conditions in \eqref{prp_tail1} and the fact that $F$ is symmetric.
Therefore, $\tail \ge N := \# \{i: X_i > x_n\}$ with high probability.
Now, $N \sim \Bin(n, p)$ where $p := (1-\eps_n)\bar{F}(x_n) + \eps_n \bar{G}(x_n-\mu_n)$, so that $N \to \infty$, since $n p \to \infty$ by the third condition in \eqref{prp_tail1}.   

Next, we show that the test is asymptotically powerless when \eqref{prp_tail2} holds.
For this, we need to show that $\tail$ is asymptotically stochastically bounded by ${\rm Geom}(1/2)$.
We do so by showing that $\tail \le L^0 + o_P(1)$, where $L^0$ is the length of the tail-run ignoring the true positive effects.  (Note that $L^0 \sim \Geom(1/2)$.) 
Under the alternative, $X_1, \dots, X_n$ may be generated as follows.  First, let $B_1,\dots,B_n$ be i.i.d.~Bernoulli with mean $\eps_n$, and then draw $X_i$ from $F$ (resp.~$G(\cdot-\mu_n)$) if $B_i = 0$ (resp.~1).   
Let $I_0 = \{i:B_i=0\}$ and $I_1 = \{i: B_i=1\}$.
By the second condition in \eqref{prp_tail2}, we have $\max_{i\in I_1} X_i \le x_n$ with probability tending to one.
Assume this is the case.
Let $N_n^- = \#\{i \in I_0: X_i < -x_n\}$ and $N_n^+ = \#\{i \in I_0: X_i > x_n\}$.  
These are binomial random variables with 
$\E N_n^\pm = n (1-\eps_n) \bar{F}(x_n) \to \infty,$
by the first condition in \eqref{prp_tail1}, so that $N_n^\pm \to \infty$ by Chebyshev's inequality. 
So, with high probability, there is an observation $X_i < 0$ such that $|X_i| > x_n$, which therefore bounds the largest positive effect.
In that case, $\tail$ is bounded by the length of the tail-run of positive signs in $\{i \in I_0: X_i > x_n\}$, which is equal to $L_0$.
We conclude that, indeed, $\tail \le L_0$ with probability tending to one.

\subsection{Moment method for analyzing a test}

We state and prove a general result for analyzing a test.  
It is particularly useful when the corresponding test statistic is asymptotically normal both under the null and alternative hypotheses.
%(Note that we were not able to use this result to analyze the $t$-test.)

\begin{lem} \label{lem:basic}
Consider a test that rejects for large values of a statistic $T_n$ with finite second moment, both under the null and alternative hypotheses.  Then the test that rejects when $T_n \ge t_n := \E_0(T_n) + \frac{a_n}2 \sqrt{\Var_0(T_n)}$ is asymptotically powerful if
\beq \label{basic1}
a_n := \frac{\E_1(T_n) - \E_0(T_n)}{\sqrt{\Var_1(T_n) \vee \Var_0(T_n)}} \to \infty.
\eeq 
Assume in addition that $T_n$ is asymptotically normal, both under the null and alternative hypotheses.  Then the test is asymptotically powerless if  
\beq \label{basic2}
\frac{\E_1(T_n) - \E_0(T_n)}{\sqrt{\Var_0(T_n)}} \to 0 \quad \text{ and } \quad \frac{\Var_1(T_n)}{\Var_0(T_n)} \to 1.
\eeq 
\end{lem}

\begin{proof}
Assume that $n$ is large enough that $a_n \ge 1$.  
%Consider the test that rejects when $T_n \ge t_n := \E_0(T_n) + a_n \sqrt{\Var_0(T_n)}$.  
By Chebyshev's inequality, the test has a level tending to zero, that is, $\P_0(T_n \ge t_n) \to 0$.  Now assume we are under the alternative.  Since
\beqn
t_n &=& \E_1(T_n) - \big(\E_1(T_n) - \E_0(T_n) - \frac{a_n}2 \sqrt{\Var_0(T_n)} \big) \\
&\le& \E_1(T_n) - \frac12 \big(\E_1(T_n) - \E_0(T_n)\big) \\
&\le& \E_1(T_n) - \frac{a_n}2 \sqrt{\Var_1(T_n)}, 
\eeqn
by Chebyshev's inequality, we see that $\P_1(T_n \ge t_n) \to 1$.  Hence, this test is asymptotically powerful.

For the second part, we have 
\[
\frac{T_n - \E_0(T_n)}{\sqrt{\Var_0(T_n)}} \rightharpoonup \cN(0,1), \text{ under the null},
\]
while
\beqn
\frac{T_n - \E_0(T_n)}{\sqrt{\Var_0(T_n)}} 
&=& \sqrt{\frac{\Var_1(T_n)}{\Var_0(T_n)}} \frac{T_n - \E_1(T_n)}{\sqrt{\Var_1(T_n)}} + \frac{\E_1(T_n) - \E_0(T_n)}{\sqrt{\Var_0(T_n)}} \\
&\rightharpoonup& \cN(0,1), \text{ under the alternative},
\eeqn
by Slutsky's theorem, since $\frac{T_n - \E_1(T_n)}{\sqrt{\Var_1(T_n)}} \rightharpoonup \cN(0,1)$ and \eqref{basic2} holds.  Hence, $\frac{T_n - \E_0(T_n)}{\sqrt{\Var_0(T_n)}}$ has the same asymptotic distribution under the two hypotheses, and consequently, is powerless at separating them.  This immediately implies that any test based on $T_n$ is asymptotically powerless.
\end{proof}

\subsection{Proof of \prpref{$t$-test}}
%Let $\sigma_{\mbox{\tiny F}}^2$ and $\sigma_{\mbox{\tiny G}}^2$ be the second and fourth moment of $F$ and $G$ respectively.  (Note that the first and third moments are equal to zero.)  Also, 
%Here, we are not able to apply \lemref{basic} directly, although the approach is similar.
Redefining $\mu_n$ as $\mu_n + {\rm mean}(G)$, we may assume that ${\rm mean}(G) = 0$ without loss of generality.
Define the sample mean and sample variance
\[\bar{X} = \frac1n \sum_{i=1}^n X_i, \qquad S^2 = \frac1{n} \sum_{i=1}^n (X_i -\bar{X})^2,\]
so that $T = \sqrt{n} \bar{X}/S$.
Under the null, $X_1, \dots, X_n$ are i.i.d.~with distribution $F$, which has finite second moment by assumption.
Hence, the central limit theorem applies and $\sqrt{n} \bar{X}$ is asymptotically normal with mean 0 and variance $\Var(F)$.
Also, 
\[S^2 = \frac1n \sum_{i=1}^n X_i^2 - \bar{X}^2 \to_P \Var(F),\]
by the law of large numbers.
Hence, by Slutsky's theorem, $T$ is asymptotically standard normal under the null.  

We now look at the behavior of $T$ under the alternative.  
The $X_i$'s are still i.i.d., with $\E_1 X_i = \eps_n \mu_n$ and 
\beq \label{varX}
\Var_1(X_i) = \eps_n (1-\eps_n) \mu_n^2 + (1-\eps_n) \Var(F) + \eps_n \Var(G) \asymp \eps_n \mu_n^2 + 1.
\eeq
Hence, by Chebyshev's inequality, 
\[\sqrt{n} \bar{X} = \sqrt{n} \eps_n \mu_n + O_P(\sqrt{\eps_n \mu_n^2 + 1}) \asymp_P \sqrt{n} \eps_n \mu_n + 1,\]
using the fact that $n \eps_n \to \infty$.
%\[\P\big(|\bar{X} - \eps_n \mu_n| > \frac{\omega_n}{\sqrt{n}} \sqrt{\eps_n \mu_n^2 + 1} \big) \to 0, \qquad \text{when } \omega_n \to \infty.\]
Let $Z_i = X_i - \eps_n \mu_n$, which are i.i.d.~with $\E Z_i = 0$ and $\E Z_i^2 = \Var_1(X_i)$.
For $k = 3 \text{ or } 4$, we have 
\[\E_1 |X_i|^k = (1-\eps_n) \E_F |X_i|^k + \eps_n \E_G |X_i+\mu_n|^k \asymp 1 + \eps_n \mu_n^k.\]
From this, it easily follows that $\E Z_i^4 \asymp 1 + \eps_n \mu_n^4$.
Since $S^2 = \frac1n \sum_i Z_i^2 - n\bar{Z}^2$, we have $\E_1 S^2 = (1-1/n) \E Z_i^2$ and 
\[\E_1 S^4 \le \frac1{n^2} \E ( \sum_i Z_i^2 )^2 \le \frac1n \E Z_i^4 + \frac{n-1}n (\E Z_i^2)^2,\]
so that
\beq \label{varS2}
\Var_1(S^2) \le \frac1n \E Z_i^4 + \frac{n-1}{n^2} (\E Z_i^2)^2 \asymp \frac1n \big(1 + \eps_n\mu_n^4 + (1 + \eps_n \mu_n^2)^2\big) \asymp \frac1n \big(1 + \eps_n\mu_n^4\big).
\eeq
Hence, by Chebyshev's inequality, 
\begin{align}
S^2 
&= (1-1/n) \Var_1(X_i) + O_P(\sqrt{\Var_1(S^2)}) \label{S2} \\
&\asymp 1 + \eps_n \mu_n^2 + O_P(\sqrt{\frac1n \big(1 + \eps_n\mu_n^4\big)}) \notag \\
&\asymp_P 1 + \eps_n \mu_n^2,\notag
\end{align}
using the fact that $n \eps_n \to \infty$.
Consequently, 
\[T \asymp_P r_n := \frac{\sqrt{n} \eps_n \mu_n + 1}{\sqrt{\eps_n \mu_n^2 + 1}}.\]
When $\sqrt{n} \eps_n \mu_n \to \infty$, we have 
\[r_n \asymp \frac{\sqrt{n} \eps_n \mu_n}{\sqrt{\eps_n \mu_n^2 \vee 1}} = \sqrt{n \eps_n} \wedge \sqrt{n} \eps_n \mu_n  \to \infty,\]
so that the test $\{T > \sqrt{r_n}\}$ has vanishing risk.
 
The same arguments show that $T$ remains bounded when $\sqrt{n} \eps_n \mu_n$ is bounded --- implying that $r_n$ is bounded --- in which case the $t$-test is not powerful.  
To prove that the $t$-test is actually powerless when $\sqrt{n} \eps_n \mu_n \to 0$ requires showing that $T$ is also asymptotically standard normal in this case.  
By the fact that $n \eps_n \to \infty$ and $\sqrt{n} \eps_n \mu_n \to 0$, we have $\eps_n \mu_n^2 \ll n \eps_n^2 \mu_n^2 \to 0$ and also $\eps_n \mu_n^4 \ll n (\eps_n \mu_n^2)^2 \ll n$.
Hence, from \eqref{varX} we get $\Var_1(X_i) \sim \Var(F)$, and from \eqref{varS2} we get $\Var_1(S^2) = o(1)$.
Therefore, on the one hand, $S^2 \to_P \Var(F)$ by \eqref{S2}.
On the other hand, Lyapunov's conditions are satisfied for $Z_i' := Z_i/\sqrt{n}$, since they are i.i.d.~with $n \E (Z_i')^2 = \E Z_i^2 = \Var_1(X_i) \to \Var(F)$ and $n \E (Z_i')^4 \asymp n (1 + \eps_n \mu_n^4)/n^2 \to 0$.
Hence, 
\[\sqrt{n} \bar{X} = \sum_{i=1}^n Z_i' + \sqrt{n} \eps_n \mu_n = \sum_{i=1}^n Z_i' + o(1),\]
is asymptotically normal with mean 0 and variance $\Var(F)$.
We conclude that $T$ is also asymptotically standard normal under the alternative when $\sqrt{n} \eps_n \mu_n \to 0$.

\subsection{Proof of \prpref{sign}}

The proof is a simple application of \lemref{basic}. 
We work with $S^+ := \sum_i \IND{\xi_{(i)} = 1}$, which is equivalent since $S = 2S^+ - n$.
Note that
\beq \label{bn}
S^+ \sim \Bin(n, 1/2 + b_n), \quad b_n := \eps_n(1/2-G(-\mu_n)).
\eeq
We have
\[
\E_0(S^+) = n/2, \quad \E_1(S^+) = (1/2 + b_n) n,
\]
and
\[
\Var_0(S^+) = n/4, \quad \Var_1(S^+) = (1/4 - b_n^2)n \leq n/4.
\]
Hence
\[
\frac{\E_1(S^+) - \E_0(S^+)}{\sqrt{\Var_0(S^+) \vee \Var_1(S^+)}} = \frac{n b_n}{\sqrt{n}/2} \to \infty,
\]
provided $\sqrt{n} b_n \to \infty$. 
By \lemref{basic}, this proves that the sign test is asymptotically powerful.

To prove that the test is powerless when the limit in \eqref{prp_sign} is 0, we first show that $S^+$ is asymptotically standard normal both under the null and under the alternative. 
The very classical normal approximation to the binomial says that
\[
\frac{S^+ - n/2}{\sqrt{n}/2} \rightharpoonup \cN(0, 1), \quad \text{under the null}.
\]
Under the alternative, we apply Lyapunov's CLT.  
The condition are easily verified: $\Var_1(S^+) = n (1/4 - b_n^2) \sim n/4$ since $b_n = o(1)$, and the variables we are summing --- here $\IND{X_i > 0}$ --- are bounded.  
Hence, it follows that $(S^+ - \E_1(S^+))/\sqrt{\Var_1(S^+)}$ is asymptotically standard normal. 
And it is easy to see that condition \eqref{basic2} holds when $\sqrt{n} b_n \to 0$.
%\[
%\frac{\E_1(S^+) - \E_0(S^+)}{\sqrt{\Var_0(S^+)}} = \frac{(1/2 + b_n) n - n/2}{\sqrt{n/4}} \asymp \sqrt{n} b_n \to 0,\]
%and
%\[
%\frac{\Var_1(S^+)}{\Var_0(S^+)} = \frac{(1/4 - b_n^2)n}{n/4} = 1 - 4 b_n^2 \to 1.
%\]
By \lemref{basic}, this proves that the sign test is asymptotically powerless.

\subsection{Proof of \prpref{wilcox}}

The proof is based on \lemref{basic}.
We work with $W^+ := \sum_i (n-i+1) \IND{\xi_{(i)} = 1}$, which is equivalent since $W = 2 W^+ - n(n+1)/2$.
The first and second moments of $W^+$ are known in closed form and, when the distribution of the variables is fixed, it is known to be asymptotically normal \citep{MR758442}.  
For completeness, and also because the distribution under the alternative changes with the sample size, we detail the proof, although no new argument is needed.

The crucial step is to represent $W^+$ as a U-statistic:
\[W^+ = \tilde W +  W^{\ddag}, \qquad \tilde W := \sum_{i=1}^n \IND{X_i > 0}, \quad W^\ddag := \sum_{1 \le i < j \le n} \IND{X_i+X_j > 0}.\]
This facilitates the computation of moments, and also the derivation of the asymptotic normality.  

Define $p_1 = \P_1(X_1 > 0)$, $p_2 = \P_1(X_1 + X_2 > 0)$, $p_3 = \P_1(X_1 > 0, X_1 + X_3 >0)$ and $p_4 = \P_1(X_1 + X_2 >0, X_1 + X_3 >0)$.
We have
\[
\E_1 W^+ = n p_1 + \frac{n(n-1)}2 p_2,
\]
and
\begin{align*}
\Var_1 W^+ &= n p_1 (1-p_1) + n(n-1)(n-2) (p_4 - p_2^2) \\
&\quad + \frac{n(n-1)}2 \big[p_2(1-p_2) + 4 (p_3 - p_1p_2)\big].
\end{align*}
Under the alternative, the $X_i$'s are i.i.d.~with distribution $Q_n(x) := (1-\eps_n)F(x) + \eps_n G(x-\mu_n)$ and density $q_n(x) := (1-\eps_n)f(x) + \eps_n g(x-\mu_n)$. 
We get
\[p_1 = \bar{Q}_n(0) = (1-\eps_n) \frac12 + \eps_n G(-\mu_n),\]
and
\begin{align*}
p_2 
&= \int \bar{Q}_n(-x) q_n(x) {\rm d}x \\
&= (1-\eps_n)^2 \int \bar{F}(-x) f(x) {\rm d}x + \eps_n^2 \int \bar{G}(-x -\mu_n) g(x-\mu_n){\rm d}x \\
& \qquad + \eps_n (1-\eps_n) \int \big(\bar{F}(-x) g(x-\mu_n) + \bar{G}(-x-\mu_n) f(x) \big) {\rm d}x \\
&= \frac12 + \frac12 \eps_n (1-\eps_n) \zeta_n + \frac12 \eps_n^2 \lambda_n,
\end{align*}
using the fact that $f$ is even and the following identities: 
\[
\int_{-\infty}^\infty \bar{F}(-x) f(x) {\rm d}x = \int_{-\infty}^\infty F(x) f(x) {\rm d}x = \frac12 F(x)^2\big|_{-\infty}^\infty = 1/2 
\]
and
\[
\int_{-\infty}^\infty \big[g(x-\mu_n) F(x) + G(x-\mu_n) f(x) \big] {\rm d}x = F(x) G(x-\mu_n) \big|_{-\infty}^\infty = 1.
\]
Similarly, we compute
\begin{align*}
p_3
= \int_0^\infty \bar{Q}_n(-x) q_n(x) {\rm d}x  
&= (1-\eps_n)^2 \int \bar{F}(-x) f(x) {\rm d}x + O(\eps_n) \\
&= \frac12 \bar{F}(-x)^2\big|_{-\infty}^\infty + O(\eps_n) = \frac38 + O(\eps_n),
\end{align*}
and
\begin{align*}
p_4 
= \int \bar{Q}_n(-x)^2 q_n(x) {\rm d}x  
&= (1-\eps_n)^3 \int \bar{F}(-x)^2 f(x) {\rm d}x + O(\eps_n) \\
&= \frac13 \bar{F}(-x)^3\big|_{-\infty}^\infty + O(\eps_n) = \frac13 + O(\eps_n).
\end{align*}
Substituting the parameters in the formulas, we obtain
\beq \label{W+mean}
\E_1 W^+ = \frac{n(n+1)}2 - n \eps_n b_n + \frac{n(n-1)}4 \big(\eps_n (1-\eps_n) \zeta_n + \eps_n^2 \lambda_n\big),
\eeq
with $b_n := 1/2 - G(-\mu_n)$, and
\beq \label{W+var}
\Var_1 W^+ = \frac{n^3}{12} + O(n^3 \eps_n).
\eeq
In particular, 
\[
\frac{\E_1 W^+ - \E_0 W^+}{\sqrt{\Var_0 W^+ \vee \Var_1 W^+}} = \sqrt{3/4} \sqrt{n} \eps_n \big((1-\eps_n) \zeta_n + \eps_n \lambda_n\big) + o(1),
\]
and $\Var_1 W^+/\Var_0 W^+ \to 1$.

Therefore, when $\sqrt{n} \eps_n (\zeta_n \vee \eps_n \lambda_n) \to \infty$, the test that rejects for large values of $W^+$ is asymptotically powerful by \lemref{basic}.  
We also note that \eqref{basic2} is satisfied when $\sqrt{n} \eps_n (\zeta_n \vee \eps_n \lambda_n) \to 0$, so it remains to show that $W^+$ is asymptotically normal.
(It is well-known that $W^+$ is asymptotically normal under the null.)  
We follow the footsteps of \cite{MR758442}.
We quickly note that $\tilde W \le n$, which is negligible compared to the standard deviation of $W^+$, which is of order $n^{3/2}$.
So it suffices to show that $W^\ddag$ is asymptotically normal.  
Its Ha\'jek projector is
\[
W^\star := \sum_{i=1}^n \E_1(W^\ddag| X_i) - (n-1) \E_1 W^\ddag,
\]
and satisfies $\E_1 W^\star = \E_1 W^\ddag$ and $\Var_1 (W^\ddag - W^\star) = \Var_1 W^\ddag - \Var_1 W^\star$.
It is easy to see that
\[
W^\star = (n-1) \sum_{i=1}^n \bar{Q}_n(-X_i) + {\rm constant},
\]
so that
\[
\Var_1 W^\star = (n-1)^2 n \Var_1 \bar{Q}_n(-X_i)  = (n-1)^2 n (p_4 - p_2^2) = \frac{n^3}{12} + O(n^3 \eps_n). 
\]
Hence, since the variables $\bar{Q}_n(-X_i)$ are bounded, Lyapunov conditions are satisfied and $W^\star$ is asymptotically normal.
Coming back to $W^\ddag$, we have $\Var_1 W^\star/\Var_1 W^\ddag \to 1$, so that
\[
\E_1 \left( \frac{W^\ddag -\E_1 W^\ddag}{\sqrt{\Var_1 W^\ddag}} - \frac{W^\star -\E_1 W^\star}{\sqrt{\Var_1 W^\ddag}} \right)^2 = 1 - \frac{\Var_1 W^\star}{\Var_1 W^\ddag} \to 0,
\]
and therefore
\[\frac{W^\ddag -\E_1 W^\ddag}{\sqrt{\Var_1 W^\ddag}} \sim_P \frac{W^\star -\E_1 W^\star}{\sqrt{\Var_1 W^\ddag}} \sim_P \frac{W^\star -\E_1 W^\star}{\sqrt{\Var_1 W^\star}} \rightharpoonup \cN(0,1).\]
We conclude that $W^\star$ is asymptotically normal also.

\subsection{Proof of \prpref{smirnov}}

We work with the form \eqref{smir}, meaning we consider the test that rejects for large values of $D_{F_n}^*$, where for a distribution function $H$, $D_H^* := \sup_{x > 0} D_{H}(x)$ with $D_H(x) := \bar{H}(x) - H(-x)$.

We already know that $\sqrt{n} D_{F_n}^* \rightharpoonup |\cN(0,1)|$ under the null hypothesis.  
  
%We have
%\begin{align}
%\bar{F}_n(x) - F_n(-x)
%& = \bar{F}_n(x) - (1-\eps_n) \bar{F}(x) - \eps_n \bar{G}(x-\mu_n) \\
%& \quad - \big[F_n(-x) - (1-\eps_n) F(-x) - \eps_n G(-x-\mu_n)\big] \\
%& \qquad + \eps_n\big[\bar{G}(x-\mu_n) - G(-x-\mu_n)\big],
%\end{align}
%where we used the fact that $F$ is symmetric.

Define $I_0$ and $I_1$ as in the proof of \prpref{tail}.  Let $N_j = |I_j|$ and $F_n^j(x) = \frac1{N_j} \sum_{i\in I_j} \IND{X_i \le x}$ for $j = 0,1$.
We have $F_n(x) = \frac{N_0}n F_n^0(x) + \frac{N_1}n F_n^1(x)$, so that
\[
\bar{F}_n(x) - F_n(-x) = \frac{N_0}n D_{F_n^0}(x) + \frac{N_1}n D_{F_n^1}(x).
\]
By the triangle inequality,
\beq \label{smirnov-proof1}
\sqrt{n} D_{F_n}^* \ge \sqrt{N_1/n} \cdot \sqrt{N_1} D_{F_n^1}^* - \big|\sqrt{N_0} D_{F_n^0}^* \big|,
\eeq
and also
\beq \label{smirnov-proof2}
\big| \sqrt{n} D_{F_n}^* - \sqrt{N_0} D_{F_n^0}^* \big| \le \big|\sqrt{N_0/n} - 1 \big|  \big|\sqrt{N_0} D_{F_n^0}^*\big| + \sqrt{N_1/n} \big|\sqrt{N_1} D_{F_n^1}^*\big|.
\eeq

For the null effects in the sample, because $F$ is symmetric and $N_0 \sim_P n \to \infty$, we have $\sqrt{N_0} D_{F_n^0}^* \rightharpoonup |\cN(0,1)|$.
For the true positive effects in the sample, by the triangle inequality,
\begin{align} 
& \big|\sqrt{N_1} D_{F_n^1}^* - \sqrt{N_1} \sup_{x>0} [\bar{G}(x-\mu_n) - G(-x-\mu_n)] \big| \label{smirnov-proof3} \\
&\qquad \le 2 \sqrt{N_1} \sup_{x \in \bbR} |F_n^1(x) - G(x-\mu_n)| = O_P(1). \notag
\end{align}
To see why the term on the RHS is bounded, we note that, given $N_1 = m$, 
\[\sqrt{N_1} \sup_{x \in \bbR} |F_n^1(x) - G(x-\mu_n)| \sim \sqrt{m} \sup_{x \in \bbR} |G_m(x) - G(x)| \rightharpoonup \Gamma,
\]
where $G_m$ denotes the empirical distribution function of an i.i.d.~sample of size $m$ drawn from $G$ and $\Gamma$ denotes the maximum absolute value of a Brownian bridge over $[0,1]$.  ($\sim$ here means ``distributed as".) 
Since $N_1 \sim_P n \eps_n \to \infty$, we infer that the same weak convergence holds unconditionally.

We now prove that the test is asymptotically powerful when the limit in \eqref{prp_smirnov} is infinite.
Under the alternative, by \eqref{smirnov-proof3} plugged into \eqref{smirnov-proof1}, we get
\[
\sqrt{n} D_{F_n}^* \ge \frac{N_1}{\sqrt{n}} \sup_{x>0} [\bar{G}(x-\mu_n) - G(-x-\mu_n)] + O_P(1) \to \infty,
\] 
where the divergence to $\infty$ is due to \eqref{prp_smirnov} diverging and the fact that $N_1 \sim_P n \eps_n$.  
Since $\sqrt{n} D_{F_n}^* = O_P(1)$ under the null, we conclude that the test is indeed powerful.

Next, we show that the test is asymptotically powerless when the limit in \eqref{prp_smirnov} is zero.
By \eqref{smirnov-proof3} plugged into \eqref{smirnov-proof1}, we get
\begin{align*}
\big| \sqrt{n} D_{F_n}^* - \sqrt{N_0} D_{F_n^0}^* \big| 
& \le \big|\sqrt{N_0/n} - 1 \big| \,  O_P(1) \\
& \quad + \frac{N_1}{\sqrt{n}} \big|\sup_{x>0} [\bar{G}(x-\mu_n) - G(-x-\mu_n)] \big| \\ 
& \quad \quad + O_P(\sqrt{N_1/n})  = o_P(1),
\end{align*}
using the fact that $N_0 \sim_P n$ and $N_1 \sim_P n \eps_n$, combined with \eqref{prp_smirnov} converging to zero.
Hence, $\sqrt{n} D_{F_n}^* \sim \sqrt{N_0} D_{F_n^0}^* \rightharpoonup |\cN(0,1)|$ under the alternative, which is the same limiting distribution as under the null.  
We conclude that the test is asymptotically powerless.

\subsection{Proof of \prpref{runs}}
\label{sec:runs-proof}
We note that our proof relies on different arguments than those of \cite{MR926417}, which are based on the classical work of \cite{MR0002083}.  
Instead, we use a Central Limit Theorem for $m$-dependent processes due to \cite{MR0350815}.
We also mention \cite{MR886465}, who tests whether independent Bernoulli random variables have the same parameter, or not. 

%Let $h_n^\pm(y) = g(\pm y - \mu_n)/f(y)$ and 
For $y \ge 0$, define
\[
p(y) = \frac{(1-\eps_n) f(y) + \eps_n g(y -\mu_n)}{2 (1-\eps_n) f(y) + \eps_n [g(y -\mu_n) + g(-y -\mu_n)]}.
\]
Let $Y_i = |X_i|$ and $\bY_n = (Y_1, \dots, Y_n)$.  Note that in the denominator in $p(y)$ is the density of $Y_1$ in model \eqref{h1}.  
Given $Y_1 = y_1, \dots, Y_n  = y_n$, the signs $\xi_{(1)}, \dots, \xi_{(n)}$ are independent Bernoulli with parameters $p_1, \dots, p_n$, where $p_i := p(y_{(i)})$ and $y_{(1)} \ge \cdots \ge y_{(n)}$ are the ordered $y$'s.  We mention that the $\xi$'s are generally {\em not} unconditionally independent.

To prove powerlessness, we use the fact that $R$ is asymptotically normal under both hypotheses and then apply \lemref{basic}.  Under the null, we saw that $R \sim \Bin(n-1, 1/2)$, and asymptotic normality comes from the classical CLT.  

Let 
\[
I_n = \int_0^\infty \frac{\eps_n^2 [g(y -\mu_n) - g(-y -\mu_n)]^2}{2 (1-\eps_n) f(y) + \eps_n [g(y -\mu_n) + g(-y -\mu_n)]} {\rm d}y.
\]
Noting that $I_n \le \eps_n (\zeta_n \wedge \eps_n \lambda_n)$ when $n$ is large enough that $\eps_n \le 1/2$, we have $\sqrt{n} I_n \to 0$ by \eqref{prp_runs}.  
It therefore exists $\omega_n \to 0$ such that $\omega_n \ge 4 I_n \sqrt{n} \vee \frac{\log n}{\sqrt{n}}$.  Define $\cY_n$ as the set of $\by_n = (y_1, \dots, y_n)$, such that
\beq \label{cY}
\sum_{i=1}^n \big(p(y_i) - \frac12\big)^2 \le \omega_n \sqrt{n}.
\eeq
Note that $I_n = \E \big[4(p(Y) - 1/2)^2\big]$, where $Y = |X|$ and $X$ is drawn from the mixture model \eqref{h1}.  Also, $4 (p(y) - 1/2)^2 \le 1$ for all $y$.  Hence, letting $A_i = 4 (p(Y_i) - 1/2)^2$, we have
\beqn
\pr{\sum_{i=1}^n A_i > 4 \omega_n \sqrt{n}}
&\le& \pr{\sum_{i=1}^n A_i - n I_n > 3 \omega_n \sqrt{n}} \\
&\le& \exp\left(- \frac{\frac92 n \omega_n^2}{4 n I_n + \frac13 (3 \omega_n \sqrt{n})} \right) \\
&\le& \exp\left(- \frac94 \omega_n \sqrt{n} \right), 
\eeqn
using the fact that $\omega_n \sqrt{n} \ge 4 n I_n$ in the first and third inequalities, and the second is Bernstein's inequality together with the fact that $\Var A_i \le \E A_i$, since $0 \le A_i \le 1$.  Hence, using the fact that $\omega_n \sqrt{n} \ge \log n$, we conclude that
\[
\P(\bY_n \notin \cY_n) \le n^{-9/4}.
\]
So it suffices to work given $\bY_n = \by_n \in \cY_n$.  Let $\P_{\by_n}$ denote the distribution of $\xi_{(1)}, \dots, \xi_{(n)}$ under model \eqref{h1} given $\bY_n = \by_n$, where $\by_n \in \cY_n$.  

Let $W_k = \{\xi_{(k)} \ne \xi_{(k-1)}\}$, so that $R = W_2 + \cdots + W_n$.  Note that $(W_k)$ forms an $m$-dependent process with $m=2$.  We apply the CLT of \cite{MR0350815} to that process.  We have $W_k \in \{0,1\}$.  Then, due to the fact that given $\bY_n=\by_n$ the $\xi$'s are independent, for $2 \le i < j \le n$ we have 
\beqn
\Var_{\by_n}(W_i + \cdots + W_j) 
&=& \sum_{k=i}^j \Var_{\by_n}(W_k) + 2 \sum_{k=i+1}^j \Cov_{\by_n}(W_k, W_{k-1}) \\
&=& \sum_{k=i}^j q_k (1-q_k) + 2 \sum_{k=i+1}^j \big(q_k^{(2)} - q_k q_{k-1}\big),
\eeqn
where 
\[
q_k := \P_{\by_n}(\xi_{(k)} \ne \xi_{(k-1)}) = p_k (1-p_{k-1}) + (1-p_{k}) p_{k-1},\] and
\beqn
q_k^{(2)} &:=& \P_{\by_n}(\xi_{(k)} \ne \xi_{(k-1)}, \xi_{(k-2)} \ne \xi_{(k-1)}) \\
&=& p_k (1-p_{k-1}) p_{k-2} + (1-p_k) p_{k-1} (1-p_{k-2}).
\eeqn
Put $a_k = p_k - 1/2$ and note that $|a_k| \le 1/2$.  We have
\[\sum_{k=i}^j q_k (1-q_k) = \frac{j-i+1}4 - 4 \sum_{k=i}^j (a_k a_{k-1})^2,\]
and
\[\sum_{k=i+1}^j \big(q_k^{(2)} - q_k q_{k-1}\big) = \sum_{k=i+1}^j a_k a_{k-2} (1 - 4 a_{k-1}^2),\]
so that
\[
\Var_{\by_n}(W_i + \cdots + W_j) \le \frac{j-i+1}4 + 2 \, \frac{j-i}4 \le j-i.
\]
We also have 
\beqn
&& \left|\frac1n \Var_{\by_n}(W_2 + \cdots + W_n) - \frac14\right| \\
&\le& \frac1{4n} + \frac{4}n \sum_{k=2}^n (a_k a_{k-1})^2 + \frac1n \sum_{k=3}^n a_k a_{k-2} (1 - 4 a_{k-1}^2) \\
&\le& \frac1{4n} + \frac1n \sum_{k=2}^n |a_k a_{k-1}| + \frac1n \sum_{k=3}^n a_k a_{k-2} \\
&\le& \frac1{4n} + \frac2n \sum_{k=1}^n a_k^2 \\
&\le& \frac1{4n} + 2 \frac{\omega_n}{\sqrt{n}} \to 0,
\eeqn
using the identity $|ab| \le (a^2+b^2)/2$ and \eqref{cY} in the last inequality.  Hence,
\[\frac1n \Var_{\by_n}(W_2 + \cdots + W_n) \to \frac14.\]
Thus the CLT of \cite{MR0350815} applies to give that $R$ is also asymptotically normal under $\P_{\by_n}$, along any sequence $\by_n \in \cY_n$.   Moreover, we also have $\Var_{\by_n}(R)/\Var_0(R) \to 1$.  For the expectation, we have 
\beqn
\E_{\by_n}(R) - \E_0(R) 
&=& \sum_{k=2}^n q_k - \frac{n-1}2 \\
&=& \sum_{k=2}^n \big(\frac12 - 2 a_k a_{k-1}\big) - \frac{n-1}2 \\
&=& - 2 \sum_{k=2}^n a_k a_{k-1},
\eeqn
and since $\Var_0(R) = (n-1)/4$, we have 
\[
\frac{\left|\E_{\by_n}(R) - \E_0(R)\right|}{\sqrt{\Var_0(R)}} \le \frac4{\sqrt{n-1}} \sum_{k=1}^n a_k^2 \le 5 \omega_n \to 0.
\]
So by \lemref{basic}, the test that rejects for small values of $R$ is asymptotically powerless.

%\medskip
%{\em The generalized Gaussian model.}  Suppose that $F=G$ is generalized Gaussian with parameter $\gamma$.  Suppose that $\mu_n = O(1)$ and define $a_n = \mu_n \gamma^{-1/\gamma}$.  Then
%\begin{align*}
%& \int \frac{[g(x -\mu_n) - g(-x -\mu_n)]^2}{f(x)} {\rm d}x \\
%&= 2 \int_0^\infty e^{-2 |x -a_n|^\gamma + |x|^\gamma} \big[1 - e^{|x +a_n|^\gamma - |x -a_n|^\gamma} \big]^2 {\rm d}x \\
%&= 2 \int_0^\infty e^{-x^\gamma + O(a_n (x \vee a_n)^{\gamma-1})} \big[1 - e^{O(a_n (x \vee a_n)^{\gamma-1})} \big]^2 {\rm d}x \\
%&\le a_n^2 \int_0^\infty e^{-\frac12 x^\gamma} O(x \vee a_n)^{2\gamma-2} {\rm d}x \asymp a_n^{2 \wedge (2\gamma+1)},
%\end{align*}
%so that $\lambda_n = O(\eps_n \mu_n ^{2 \wedge (2\gamma+1)})$ as claimed.

\subsection{Proof of \prpref{longest}}

We keep the same notation as in the previous section, except we redefine $\cY_n$ as the set of $\by_n = (y_1, \dots, y_n)$ such that $\max_i y_i \le x_n$.  Equivalently, $\cY_n = [0,x_n]^n$.

We first prove that the test is asymptotically powerless under \eqref{prp_longest1}-\eqref{prp_longest2}.
First, by the union bound, 
\begin{align*}
\P(\bY_n\notin \cY_n) 
&= \P(\max_i Y_i > x_n) \\
&\le n \big[2 (1-\eps_n) \bar F(x_n) + \eps_n G(-x_n-\mu_n) + \eps_n \bar G(x_n -\mu_n)\big] \to 0,
\end{align*}
because of \eqref{prp_longest1}.
Therefore, we work given $\bY_n= \by_n \in \cY_n$ as before.

Let $p_n^* = \max\{p(y) : 0 \le y \le x_n\}$ and note that
\[
p_n^* \le \frac1{2 - \eta_n}, \quad \eta_n := \eps_n \max_{0 \le y \le x_n} \frac{(g(y -\mu_n) - g(-y -\mu_n))_+}{(1-\eps_n) f(y) + \eps_n g(y -\mu_n)}.
\]
%as soon as $\eta_n \le 1/2$, which happens eventually since $\eta_n \log n \to 0$ by \eqref{prp_longest2}.  
When $\by_n \in \cY_n$, we have that $p_i \le p_n^*$ for all $i$.  

Let $L_{n,p}$ denote the length of the longest-run in a sequence of $n$ i.i.d. Bernoulli random variables with parameter $p$.  Also, let $Z_p$ have the distribution $\P(Z_p \le z) = \exp(-p^z)$.  From \citep[Ex.~3]{MR972770}, we have the weak convergence
\[
L_{n,p} - \frac{\log(n(1-p))}{\log (1/p)} \rightharpoonup \lfloor Z_p + r \rfloor - r, 
\]
when $n \to \infty$ along a sequence such that $\log(n (1-p))/\log(1/p) \to r \mod 1$.  

Now, under the null, $L$ has the distribution of $L_{n,1/2}$.  Under $\P_{\by_n}$, $L$ is stochastically bounded by $L_{n,p_n^*}$.  In fact, $L_{n,p_n^*}$ is itself stochastically bounded by $L_{n,1/2}$ in the limit.  Indeed, on the one hand, we have
\begin{align*}
\frac{\log (n(1-p_n^*))}{\log (1/p_n^*)} - \frac{\log n}{\log 2} 
&\le \frac{(\log n) \log (2/(2-\eta_n))}{(\log 2) \log(2- \eta_n)} \\
&\sim \frac{(\log n) \eta_n/2}{(\log 2)^2} = O(\eta_n \log n) \to 0,
\end{align*}
by \eqref{prp_longest2}; on the other hand, $Z_{p_n^*}$ is stochastically bounded by $Z_{\frac12+\eta_n}$, which converges to $Z_{1/2}$ in distribution.  We therefore conclude that the test is asymptotically powerless.

We now prove the asymptotic powerfulness of the test under either \eqref{prp_longest3}, or \eqref{prp_longest4}-\eqref{prp_longest5}.  
In Case (i), we quickly note that \eqref{prp_longest3} is identical to \eqref{prp_tail1} except for the log factor in the rightmost condition, and the exact same arguments showing that the tail-run test is asymptotically powerful under \eqref{prp_tail1} imply that, under \eqref{prp_longest3}, $L^\ddag \gg \log n$, where $L^\ddag$ is the tail run defined in \eqref{tail}.  Hence, under the alternative, $L \ge L^\ddag \gg \log n$, compared to $L \sim L_{n,1/2} = O_P(\log n)$ under the null.
%The technical details are postponed to the proof of \prpref{tail}.

In Case $(ii)$, \eqref{prp_longest1} holds, so that we may work given $\bY_n= \by_n \in \cY_n$ as before, and the arguments are almost the same as when we proved powerlessness, but in reverse.
Let $I_n = [x'_n, x'_n+c]$.  
Redefine $p_n^* = \min \{p(y) : y \in I_n\}$ and note that
\[
p_n^* \ge \frac1{2 - \eta_n}, \quad \eta_n := \eps_n \min_{y \in I_n} \frac{g(y -\mu_n) - g(-y -\mu_n)}{(1-\eps_n) f(y) + \eps_n g(y -\mu_n)},
\]
as soon as $\eta_n > 0$.
In fact, by \eqref{prp_longest4}, $\eta_n \ge a > 0$ so that $p_n^* \ge \frac1{2 - a} > \frac12$.
  
We have that $L \ge L'$, where $L'$ is the longest-run of pluses among $\{\xi_{(i)}: y_{(i)} \in I_n\}$.
The number of $y_i$'s falling in $I_n$, denoted by $N_n$, is stochastically larger than $\Bin(n, q_n)$, where 
\begin{align*}
q_n 
&:= 2(1-\eps_n) \big(F(x_n'+c) - F(x_n')\big) \\
& \qquad + \eps_n \big(G(x'_n+c-\mu_n) - G(x'_n-\mu_n)\big) \\
&\ge c \min_{y \in I_n} \big( 2 f(y) + \eps_n g(y -\mu_n)\big) \ge c d n^{-b},
\end{align*}
where the last inequality holds eventually due to \eqref{prp_longest5}.
Therefore, with high probability as $n\to \infty$, $N_n \ge (cd/2) n^{1-b}$.
Given this is the case, $L'$ is stochastically bounded from below by $L_{(cd/2) n^{1-b},p_n^*}$, and we know that 
\begin{align*}
L_{cd n^{1-b},p_n^*} 
&\ge \frac{\log((cd/2) n^{1-b}(1-p_n^*))}{\log(1/p_n^*)} + O_P(1) \\
&\ge \frac{(1-b) \log n + \log((cd/2) (1-a)/(2-a))}{\log(2-a)} + O_P(1) \\
&= \frac{(1-b) \log n}{\log(2-a)} + O_P(1).
\end{align*}
We compare this with the size of $L$ under the null, which is $\frac{\log(n)}{\log(2)} + O_P(1)$:
\[
\frac{(1-b) \log n}{\log(2-a)} - \frac{\log n}{\log 2} = \left(\frac{1-b}{\log(2-a)} - \frac1{\log 2}\right) \log n \to \infty, 
\]
since the constant factor is positive by the upper bound on $b$.

\end{document}